\documentclass[11pt]{article}
\usepackage{xcolor}
\usepackage{amsthm}
\usepackage{amsfonts}
\usepackage{graphicx}
\usepackage{subfig}
\usepackage{latexsym}
\usepackage{url}
\usepackage{bbm}
\usepackage{mathtools}
\mathtoolsset{showonlyrefs}

\newcommand{\CC}{\mathbb{C}}
\newcommand{\R}{\mathbb{R}}

\newcommand{\E}{\mathbb{E}}

\renewcommand{\P}{\mathbb P}
\renewcommand{\H}{\mathbb H}

\newcommand{\al}{\alpha}
\newcommand{\be}{\beta}
\newcommand{\la}{\lambda}

\newcommand{\ga}{\gamma}

\newcommand{\ep}{\varepsilon}

\newcommand{\de}{\delta}
\newcommand{\te}{\theta}

\newcommand{\ke}{\mathcal K}
\newcommand{\ia}{\mathcal I}
\newcommand{\n}{\mathcal N}

\newcommand{\ca}{\mathcal C}

\newcommand{\f}{\mathcal F}

\newcommand{\lll}{\mathcal L}

\newcommand{\proba}{(\Omega ,\mathcal{F},(\f_t)_{t\geq0},\P)}

\newcommand{\schw}{\stackrel{\raisebox{-1pt}{\textup{\tiny d}}}{\longrightarrow}}

\newcommand{\toas}{\stackrel{\mbox{\upshape\tiny a.s.}}{\longrightarrow}}

\newcommand{\wC}{\widehat{C}}

\newcommand{\wA}{\widehat{A}}

\DeclareMathOperator*{\argmin}{argmin}

\newcommand{\lan}{\langle}
\newcommand{\ran}{\rangle}

\newcommand{\bee}{\begin{equation}}
\newcommand{\eee}{\end{equation}}
\newcommand{\beea}{\begin{array}}
\newcommand{\eeea}{\end{array}}

\renewcommand{\theequation}{\arabic{section}.\arabic{equation}}

\theoremstyle{plain}
\newtheorem{prop}{Proposition}[section]
\newtheorem{cor}[prop]{Corollary}
\newtheorem{theo}[prop]{Theorem}
\newtheorem{lem}[prop]{Lemma}

\theoremstyle{definition}

\begin{document}

\title{On Dantzig and Lasso estimators of the drift in a high dimensional Ornstein-Uhlenbeck model \thanks{The authors
gratefully acknowledge financial support of ERC Consolidator Grant 815703
``STAMFORD: Statistical Methods for High Dimensional Diffusions''.}} 
\author{Gabriela Cio\l{}ek\thanks{Department
of Mathematics, University of Luxembourg, 
E-mail: gabriela.ciolek@uni.lu.}  \and
Dmytro Marushkevych \thanks{Department
of Mathematics, University of Luxembourg, 
E-mail: dmytro.marushkevych@uni.lu.} \and
Mark Podolskij\thanks{Department
of Mathematics, University of Luxembourg,
E-mail: mark.podolskij@uni.lu.}}

\maketitle

\begin{abstract}
\noindent In this paper we present new theoretical results for the Dantzig and Lasso estimators of the drift in 
a high dimensional Ornstein-Uhlenbeck model under sparsity constraints. 
Our focus is on oracle inequalities for both estimators and error bounds 
with respect to several norms. In the context of the Lasso estimator our paper is strongly related to \cite{GM19}, 
who investigated the same problem under row sparsity. We improve their rates and also prove the restricted eigenvalue
property solely under ergodicity assumption on the model. Finally, we demonstrate a numerical analysis to uncover the finite
sample performance of  the Dantzig and Lasso estimators.

\ \

\bigskip


\end{abstract}

\section{Introduction} \label{sec1}
\setcounter{equation}{0}
\renewcommand{\theequation}{\thesection.\arabic{equation}}

During past decades an immense progress has been achieved in statistics for stochastic processes. Nowadays, comprehensive studies on statistical inference for diffusion processes under low and high frequency observation schemes can be found in monographs \cite{JP12,KS97,K04}. Most of the existing literature is considering a fixed dimensional parameter space, while a high dimensional framework received much less attention in the diffusion setting. 

Since the pioneering work of McKean \cite{M66, M67}, high dimensional diffusions entered the scene in the context of modelling the movement of gas particles. More recently, they found  numerous  applications in economics and biology, among other disciplines
\cite{BCC11, CZ16,FTC09}. Typically, high dimensional diffusions are studied in the framework of  \textit{mean field theory}, which aims at 
bridging the interaction of particles at the microscopic scale and the mesoscopic features of the system (see e.g. \cite{S91} for a mathematical study).
In physics particles are often assumed to be statistically  equal,  but this homogeneity assumption is not appropriate in other applications. 
For instance, in \cite{CZ16} high dimensional SDEs  are used to model the wealth of trading agents in an economy, who are often far from being equal in their trading behaviour. Another example is the flocking phenomenon of individuals \cite{BCC11}, where it seems natural 
to assume that there are only very few ``leaders'' who have a distinguished role in the community. These examples motivate to investigate statistical inference for diffusion processes under sparsity constraints.

This paper is focusing on  statistical analysis of a $d$-dimensional Ornstein-Uhlenbeck model of the form 
\bee \label{OU}
dX_t= -A_0 X_t dt + dW_t, \qquad t \geq 0,
\eee
defined on a filtered probability space $\proba$, with underlying observation $(X_t)_{t \in [0,T]}$. Here $W$ denotes a standard $d$-dimensional Brownian motion and $A_0 \in \R^{d\times d}$ represents the unknown interaction matrix. Ornstein-Uhlenbeck processes are one of the most basic parametric diffusion models. When the dimension $d$ is fixed and $T\to \infty$, statistical estimation of the parameter $A_0$ has been discussed in several papers. Asymptotic analysis of the maximum likelihood estimator in the ergodic case can be found in e.g. \cite{KS97} while investigations of the non-ergodic setting can be found in \cite{KR01,KS99}. The adaptive Lasso estimation for multivariate diffusion models has been investigated in \cite{DI12}.  

Our main goal is to study the estimation of $A_0$ under sparsity constraints in the large $d$/large $T$ setting. Such a mathematical problem 
finds its main motivation in the analysis of bank connectedness whose wealth is modelled by the diffusion process $X$. This field of economics, which studies linkages between a large number of banks associated with e.g. asset/liability positions and contractual relationships, is key to understanding systemic risk in a global economy \cite{J08a}. Typically, the connectivity structure, which is represented by the parameter $A_0$, is quite sparse
since only few financial players are significant in an economy, and the main focus is on estimation of non-zero components of $A_0$.
 
Theoretical results in the high dimensional diffusion setting are rather scarce. In this context we would like to mention the Dantzig selector which was introduced in \cite{Can07} and primarily designed for linear regression models. More specifically, \cite{Can07} established sharp non-asymptotic bounds on the $l_2$-error in the estimated coefficients and proved that the error is within a factor of $log(d)$ of the error that would have been reached if the locations of the non-zero coefficients were known. Further extensions of the aforementioned results can be found in \cite{F19} and \cite{PI14}, which study the Dantzig selector for discretely observed linear diffusions and support recovery for the drift coefficient, respectively.  Our work is closely related to the recent article \cite{GM19}, where estimation of  $A_0$ under row sparsity has been investigated. The authors propose to use the classical Lasso approach and derive upper and lower bounds for the estimation error. We build upon their analysis and provide oracle inequalities and non-asymptotic theory for the the Lasso and Dantzig estimators. In comparison to  \cite{GM19}, we obtain an improved upper bound for the Lasso estimator, which essentially matches the theoretical lower bound, and also show that the \textit{restricted eigenvalue property} is automatically satisfied under ergodicity condition on the model \eqref{OU} (in \cite{GM19} the extra assumption (H4) has been imposed). The latter is proved via Malliavin calculus methods proposed in
\cite{NV09}.  Moreover, we show that the Lasso and Dantzig estimators are asymptotically efficient, which is a well known fact in linear regression models (cf. \cite{BRT09}). Finally, we present a simulation study to uncover the finite sample properties of both estimators.

The paper is organised as follows. Section \ref{sec2} is devoted to the exposition of the classical estimation theory in the fixed dimensional setting 
and to definition of the Lasso and Dantzig estimators. Concentration inequalities for various stochastic terms are derived in Section \ref{sec3}. In particular, we show the restricted eigenvalue property under the ergodicity assumption via Malliavin calculus methods. In Section \ref{sec4} we present oracle inequalities and error bounds for both estimators. Numerical simulation results are demonstrated in Section \ref{sec5}. Finally, some proofs are collected in Section \ref{sec6}.

\section{The model, notation and main definitions} \label{sec2}
\setcounter{equation}{0}
\renewcommand{\theequation}{\thesection.\arabic{equation}}

\subsection{Notation} \label{sec2.1} 
In this subsection we briefly introduce the main notations used throughout the paper. For a vector or a matrix $x$ the transpose of $x$ is denoted by $x^{\top}$. For $p\geq 1$ and $A\in \R^{d_1 \times d_2}$, we define the $l_p$-norm as
\[
\|A\|_p:= \left(\sum_{1\leq i \leq d_1, 1\leq j \leq d_2} |A_{ij}|^p \right)^{1/p}. 
\] 
We denote by $\|A\|_{\infty}=\lim_{p\to \infty} \|A\|_p$ the maximum norm and set $\|A\|_{0}:=\sum_{1\leq i \leq d_1, 1\leq j \leq d_2} 1_{\{A_{ij} \not =0\}}$. We associate to the Frobenius norm $\|\cdot\|_2$ the scalar product
\[
\lan A_1, A_2 \ran_{\text{F}} := \text{tr} (A_1^{\top} A_2), \qquad    A_1,A_2\in \R^{d_1 \times d_2},
\]
where $\text{tr}$ denotes the trace.  For a symmetric matrix $A \in \R^{d \times d}$ we write $\la_{\max}(A)$, $\la_{\min}(A)$ for the largest and the smallest eigenvalue of $A$, respectively. 
We denote by $\|A\|_{\text{op}}:=\sqrt{\la_{\max}(A^{\top} A)}$ the operator norm of  $A \in \R^{d \times d}$. For any $J \subset\{1,\ldots, d_1\}
\times \{1,\ldots, d_2\}$ and $A\in \R^{d_1 \times d_2}$, the matrix $A_{|J}$ is defined via
\bee \label{restA}
(A_{|J})_{ij}:= A_{ij} 1_{(i,j) \in J}. 
\eee
For a quadratic matrix $A \in \R^{d \times d}$, $\text{diag}(A)$ stands for the diagonal matrix satisfying $\text{diag}(A)_{ii}=A_{ii}$. We also  introduce the notation 
\bee \label{Csc0}
\ca(s,c_0):= \left\{ A \in \R^{d \times d}\setminus\{ 0 \} : ~\|A\|_1 \leq (1+c_0) \|A_{| \ia_s(A)}\|_1 \right\},
\eee
where $c_0>0$ and $\ia_s(A)$ is a set of coordinates of $s$ largest elements of $A$. 
Furthermore, $\text{vec}$ denotes the vectorisation operator and $\otimes$ stands for the Kronecker product. 
For $z\in \CC$ we denote by $\mathfrak{Re}(z)$ (resp. $\mathfrak{Im}(z)$) the real (resp. imaginary) part of $z$. 
Finally, for stochastic processes $(X_t)_{t\in [0,T]}, (Y_t)_{t\in [0,T]} \in L^2([0,T],dt)$ we introduce the scalar product
\[
\lan X, Y \ran_{L^2} := \frac 1T \int_0^T X_t Y_t dt. 
\]

\subsection{The setting and fixed dimensional theory} \label{sec2.2} 
We consider a $d$-dimensional Ornstein-Uhlenbeck process introduced in \eqref{OU}. 
Throughout this paper the matrix $A_0$  is assumed to satisfy the 
following condition: \\ \\
(H) Matrix $A_0$ is diagonalisable with eigenvalues $\te_1,\ldots,\te_d \in \CC$, i.e.
\[
A_0 = P_0 \text{diag}(\te_1,\ldots,\te_d) P_0^{-1},
\]  
where the column vectors of $P_0$ are eigenvectors of $A_0$. Furthermore, the  eigenvalues $\te_1,\ldots,\te_d \in \CC$ have strictly positive real 
parts:
\bee \label{eigenmin}
\mathfrak{r}_{0}:=\min_{1\leq j\leq d}(\mathfrak{Re}(\te_i))>0.
\eee
It is well known that under condition (H) the stochastic differential equation \eqref{OU} exhibits a unique stationary solution, which can be written explicitly as
\bee
X_t = \int_{-\infty}^t \exp\left(-(t-s) A_0\right) dW_t.
\eee
In this case we have that 
\bee
X_t \sim \n(0, C_{\infty}) \qquad \text{with} \qquad C_{\infty}:= \int_0^{\infty} \exp(-sA_0) \exp(-sA_0^{\top}) ds.
\eee
We assume that  the complete path $(X_t)_{t\in[0,T]}$ is observed and we are interested in estimating the unknown parameter $A_0$. Let us briefly recall the classical maximum likelihood theory when $d$ is fixed and $T\to \infty$. When $\P_{A}^T$ denotes the law of the process \eqref{OU} with transition matrix $A$ restricted to $\f_{T}$, the log-likelihood function is explicitly computed via Girsanov's theorem as 
\bee \label{loglike}
\log(\P_{A}^T/ \P_{0}^T) = -\int_0^T(AX_t)^{\top}dX_t-\frac{1}{2}\int_0^T(AX_t)^{\top}(AX_t)dt.
\eee
Consequently, the maximum likelihood estimator $\wA_{\text{ML}}$ is given by
\bee
\wA_{\text{ML}}= -\Big( \int_0^T dX_t X_t^{\top}\Big) \Big( \int_0^T X_t X_t^{\top}dt\Big)^{-1}.
\eee
Under condition (H) the estimator $\wA_{\text{ML}}$ is asymptotically normal, i.e.
\bee \label{MLclt}
\sqrt{T} \left(\text{vec}(\wA_{\text{ML}}) - \text{vec}(A_0) \right) \schw \n_{d^2}\left(0, C_{\infty}^{-1} \otimes \text{id} \right)
\eee 
with $\text{id}$ denoting the $d$-dimensional identity matrix.  Indeed, we have the identity $\wA_{\text{ML}} - A_0 = -\ep_T \wC_T^{-1}$
with
\bee \label{epCT}
\ep_T := \frac 1T  \int_0^T dW_t X_t^{\top} \qquad \text{and} \qquad \wC_T:= \frac 1T \int_0^T X_t X_t^{\top}dt \toas C_{\infty},
\eee
and the result \eqref{MLclt} follows from the standard martingale central limit theorem. We refer to \cite[p. 120--124]{KS97} for a more detailed exposition. 

When assumption (H) is violated the asymptotic theory for the maximum likelihood estimator $\wA_{\text{ML}}$ is more complex. If some eigenvalues $\te_j$ satisfy  $\mathfrak{Re}(\te_i)<0$ exponential rates appear as it has been shown in  \cite{KS99}. A further application of Ornstein-Uhlenbeck processes to co-integration is discussed in \cite{KR01}, where the condition  $\mathfrak{Re}(\te_i)=0$ appears for some $i$'s.

\subsection{The Lasso and Dantzig estimators} \label{sec2.3} 

Now we turn our attention to large $d$/large $T$ setting. We consider the Ornstein-Uhlenbeck model \eqref{OU} satisfying the assumption (H) and assume that the unknown transition matrix $A_0$ satisfies the constraint
\bee \label{constraint} 
\|A_0\|_0 \leq s_0. 
\eee
We remark that due to condition \eqref{eigenmin} it must necessarily hold that $s_0\geq d$. A standard approach to estimate $A_0$ under the sparsity constraint \eqref{constraint} is the Lasso method, which has been investigated in \cite{GM19} in the framework of an Ornstein-Uhlenbeck model.  
The Lasso estimator is defined as
\bee \label{lasso}
\wA_{\text{L}}:= \argmin_{A\in \R^{d\times d}} \left( \lll_T(A) + \la \|A\|_1\right) \textit{\;\;with}\qquad \lll_T(A):=-\frac 1T \log(\P_{A}^T/ \P_{0}^T),
\eee
where $\la>0$ is a tuning parameter.  We remark that $\wA_{\text{L}}$ can be computed efficiently, since it is a solution of a convex optimisation problem.

Next, we are going to introduce the Dantzig estimator of the parameter $A_0$. According to \eqref{loglike} the quantity $\lll_T(A)$ can be written as 
\bee \label{LTformula}
\lll_T(A) =\text{tr} \left( (\ep_T-A_0 \wC_T )A^{\top} + \frac 12 A \wC_T A^{\top}  \right) \qquad \text{and} \qquad \nabla \lll_T(A) = \ep_T-A_0 \wC_T
+A \wC_T.
\eee
We recall that $B$ belongs to a subdifferential of a convex function $f:\R^{d \times d} \to \R$ at point $B_0$, $B\in \partial f(B_0)$, if 
$\lan B, A-B_0 \ran_{\text{F}} \leq f(A)-f(B_0)$ for all $A \in \R^{d \times d}$. In particular, $B\in \partial \|B_0\|_1$ satisfies the constraint  
$\|B\|_{\infty} \leq 1$.   A necessary and sufficient condition for the minimiser at \eqref{lasso} is the fact that $0$ belongs to the subdifferential of the
function $A \mapsto \lll_T(A) + \la \|A\|_1$. This implies that the Lasso estimator $\wA_{\text{L}}$ satisfies the constraint
\bee \label{Lassoconstr}
\|\wA_{\text{L}} \wC_T + \ep_T-A_0 \wC_T \|_{\infty} \leq \la.
\eee
Now, the Dantzig estimator $\wA_{\text{D}}$ of the parameter $A_0$ is defined as a matrix with the smallest $l_1$-norm that satisfies the inequality
\eqref{Lassoconstr}, i.e.
\bee \label{Dantzig} 
\wA_{\text{D}} := \argmin_{A\in \R^{d\times d}} \left\{\|A\|_1:~ \|A \wC_T + \ep_T-A_0 \wC_T \|_{\infty} \leq \la\right\}.
\eee    
By definition of the Dantzig estimator we have that $\|\wA_{\text{D}}\|_1 \leq \|\wA_{\text{L}}\|_1$. In particular, when the tuning parameters $\lambda$ for Lasso and Dantzig estimators are preset to be the same, then the Lasso estimate is always a feasible solution to the Dantizg selector minimization problem although it may not necessarily be the optimal solution. This implies, that when respective solutions are not identical, the Dantizg selector solution is sparser (in $l_1$- norm) than the Lasso solution (see \cite{Jam09}, Appendix A for details). From the computational point view,  the Dantzig estimator can be found numerically via linear programming for convex optimisation with constraints. 

The following basic inequality, which is a direct consequence of the fact that $ \lll_T(\wA_{\text{L}}) + \la \|\wA_{\text{L}}\|_1 \leq  \lll_T(A) + \la \|A\|_1$ 
for all $A \in \R^{d\times d}$, provides the necessary basis for the analysis of the error $\wA_{\text{L}}-A_0$.

\begin{lem} \label{lem1}
(\cite[Lemma 3]{GM19}) For any $A \in \R^{d\times d}$ and $\la>0$ it holds that 
\[
\|(\wA_{\text{L}}-A_0) X\|_{L^2}^2 - \|(A-A_0) X\|_{L^2}^2 \leq 2 \lan \ep_T, A- \wA_{\text{L}} \ran_{\text{F}} - 
\|(A-\wA_{\text{L}}) X\|_{L^2}^2 +2\la (\|A\|_1- \|\wA_{\text{L}}\|_1),
\]
where the quantity $\ep_T$ is defined in  \eqref{epCT}. 
\end{lem}
From Lemma \ref{lem1} it is obvious that we require a good control over martingale term $\lan \ep_T, V\ran_{\text{F}}$
for certain matrices $V\in \R^{d\times d}$ to get an upper bound on the prediction error $\|(\wA_{\text{L}}-A_0) X\|_{L^2}$. Another important ingredient is the \textit{restricted eigenvalue property}, which is a standard requirement in the analysis of Lasso estimators (see e.g. \cite{BRT09,BV}). In our setting the   restricted eigenvalue property amounts in showing that 
\[
\inf_{V \in \ca(s,c_0)} \frac{\|VX\|_{L^2}^2}{\|V\|^2_{2}} \text{ is bounded away from } 0 \text{ with high probability.}
\]
Interestingly, the latter is a consequence of the model assumption (H) and not an extra condition as in the framework of linear regression. This has been noticed in \cite{GM19}, but an additional condition (H4) was required which is in fact not needed as we will show in the next section.  

In order to establish the connection between the Dantzig and the Lasso estimators we will show the inequality 
\[
\left| \|(\wA_{\text{D}}-A_0) X\|_{L^2} - \|(\wA_{\text{L}}-A_0) X\|_{L^2}\right| \leq c  \|\wA_{\text{L}}\|_0 \la^2
\]
for a certain constant $c>0$, which holds with high probability. Once  the term $\|\wA_{\text{L}}\|_0$ is controlled, we deduce statements about the error term $\wA_{\text{D}}-A_0$ via the corresponding analysis of $\wA_{\text{L}}-A_0$.

\section{Concentration bounds for the stochastic terms} \label{sec3}
\setcounter{equation}{0}
\renewcommand{\theequation}{\thesection.\arabic{equation}}

In this section we derive various concentration inequalities, which play a central role in the analysis of the estimators $\wA_{\text{L}}$ and 
$\wA_{\text{D}}$.

\subsection{The restricted eigenvalue property} \label{sec3.1}
This subsection is devoted to the proof of the restricted eigenvalue property. The main result of this subsection relies heavily on some theoretical techniques presented in \cite{NV09}, where Malliavin calculus is applied in order to obtain tail bounds for certain functionals of Gaussian processes. In the following,  we introduce some basic notions of Malliavin calculus; we refer to the monograph \cite{N06} for a more detailed exposition.   

Let $\H$ be a real separable Hilbert space. We denote by $B =
\{B(h):~h \in \H\}$ an \textit{isonormal Gaussian process} over
$\H$. That is, $B$ is a centred  Gaussian family with covariance kernel given by
\bee\label{isometry}
\E\big[B(h_1)B(h_2)\big]=\lan h_1,h_2\ran_{\H}.
\eee
We shall use the notation $L^2(B)= L^2(\Omega,\sigma(B),\P)$. For every $q\geq 1$, we write
$\H^{\otimes q}$ to indicate the $q$th tensor product of $\H$;\; $\H^{\odot q}$ stands for the symmetric $q$th tensor. 
We denote by $I_q$ the isometry between $\H^{\odot q}$ and the $q$th Wiener chaos of
$X$. It is well-known (see e.g. \cite[Chapter 1]{N06}) that any random variable $F\in L^2(B)$ admits the \textit{chaotic expansion}
\bee \label{ChaosExpansion}
F=\sum_{q=0}^\infty I_q(f_q), \qquad I_0(f_0):=\E[F],
\eee
where the series converges in $L^2$ and the kernels $f_q\in\H^{\odot q}$ are uniquely
determined by $F$. The operator $L$, called the \textit{generator of the Ornstein-Uhlenbeck semigroup}, is defined as
\[
LF:= - \sum_{q=1}^\infty qI_q(f_q)
\]  
whenever the latter series converges in $L^2$. The pseudo inverse $L^{-1}$ of $L$ is defined by $L^{-1}F=- \sum_{q=1}^\infty q^{-1}I_q(f_q)$.

Next, let us denote by $\mathcal{S}$  the set of all smooth cylindrical random
variables of the form 
$
F = f\big(B(h_1), \ldots, B(h_n)\big),
$
where $n\geq 1$, $f : \R^n \rightarrow \R$ is a
$C^{\infty}$-function with compact support and $h_i\in\H$. The
Malliavin derivative $DF$ of $F$ is defined as
$$DF:= \sum_{i =1}^n \frac{\partial f}{\partial x_i}\big(B(h_1), \ldots, B(h_n)\big)
h_i.$$ 
The space $\mathbb{D}^{1,2}$ denotes the
closure of $\mathcal{S}$ with respect to norm 
$\| F\|_{1,2}^2 : = \E[F^2] +
\E[ \| D F\|_{\H}^2].
$
The Malliavin derivative $D$ verifies the following \textit{chain rule}: when
$\varphi:\R^n\rightarrow\R$ is in $C^1_b$ (the
set of continuously differentiable functions with bounded partial derivatives) and if $(F_i)_{i=1,\ldots,n}$ is a vector of
elements in $\mathbb{D}^{1,2}$, then
$\varphi(F_1,\ldots,F_n)\in\mathbb{D}^{1,2}$ and
$$
D\varphi(F_1,\ldots,F_n)=\sum_{i=1}^n
\frac{\partial\varphi}{\partial x_i} (F_1,\ldots, F_n)DF_i.
$$
The next theorem establishes left and right tail bounds for certain elements $Z \in \mathbb{D}^{1,2}$.

\begin{theo} \label{Mallbound}
(\cite[Theorem 4.1]{NV09})
Assume that $Z \in \mathbb{D}^{1,2}$ and define the function 
$$g_Z(z):=\E[\lan DZ,- DL^{-1} Z\ran_{\H}|~Z=z].$$ 
Suppose that the following conditions hold for some $\al \geq 0$ and $\be>0$:
\begin{itemize}
\item[(i)] $g_Z(Z)\leq \al Z + \be$ holds $\P$-almost surely,
\item[(ii)] The law of $Z$ has a Lebesgue density. 
\end{itemize}
Then, for any $z>0$, it holds that 
\[
\P(Z\geq z) \leq \exp\left( - \frac{z^2}{2\al z+2\be}\right) \qquad \text{and} \qquad 
\P(Z\leq -z) \leq \exp\left( - \frac{z^2}{2\be}\right).
\]
\end{theo} 

Now, we apply Theorem \ref{Mallbound} to certain quadratic forms of the Ornstein-Uhlenbeck process $X$. The following result is crucial 
for proving the restricted eigenvalue property.

\begin{prop} \label{PropH4}
Suppose that assumption (H) is satisfied and let $\wC_T$ be defined as in \eqref{epCT}. Then it holds for all $x>0$:
\bee \label{ineqH}
\sup_{v\in \R^d:~\|v\|_2=1} \P\left( | v^{\top} ( \wC_T-C_\infty ) v|\geq x\right) \leq 2\exp{(-T H_0(x))},
\eee
where the function $H_0$ is defined as
\bee
H_0(x)=\frac{\mathfrak{r}_0}{8\mathfrak{p}_0\mathfrak{K}_\infty} \frac{x^2}{x+\mathfrak{K}_\infty}
\eee
with $\mathfrak{K}_\infty=\la_{\max}(C_\infty)$ and $\mathfrak{p}_0=\|P_0\|_{\text{\rm op}}\|P_0^{-1}\|_{\text{\rm op}}$, and the quantities $P_0$
and $\mathfrak{r}_0$ are introduced in assumption (H). 
\end{prop}

\begin{proof}
We define the centred stationary Gaussian process $Y_t^v= v^{\top} X_t$ and note that its covariance kernel is given by
$\E[Y_t^v Y_s^v]=\rho_v(|t-s|)$ with $\rho_v(r):=v^{\top}\exp(-r A_0)C_\infty v$. By submultiplicativity of the operator norm we conclude that
\[
|\rho_v(r)| \leq \|\exp(-r A_0)\|_{\text{op}}\|C_\infty\|_{\text{op}}\leq \exp(-\mathfrak{r}_0 r) \mathfrak{p}_0 \mathfrak{K}_\infty. 
\]  
We observe that $(Y_t^v)_{t \in [0,T]}$ can be considered as an isonormal Gaussian process indexed by a separable Hilbert space $\H$ whose scalar product is induced by the covariance kernel of $(Y_t^v)_{t \in [0,T]}$. In particular, we can write $Y_t^v=B(h_t)$ and 
$\lan h_t, h_s \ran_{\H}= \rho_v(|t-s|)$. We introduce the quantity
\[
Z_T^v:= v^{\top} ( \wC_T-C_\infty ) v = \frac 1T \int_0^T (Y_t^v)^2 - \E[(Y_t^v)^2] dt
\]
and notice that $Z_T^v$ is an element of the second order Wiener chaos. Hence, $Z_T^v$ has a Lebesgue density and we have $L^{-1}Z_T^v= -Z_T^v/2$, and we conclude by the chain rule
that 
\begin{align*}
&\lan D Z_T^v, - D L^{-1} Z_T^v \ran_{\H}=\frac{1}{2}\|D Z_T^v\|_{\H}^2 \leq \frac{2}{T^2}\int_0^T\int_0^T |Y_t^v Y_s^v| |\rho_v(t-s)|dt ds\\[1.5 ex]
&\leq \frac{2}{T^2} \int_0^T\int_0^T (Y_t^v)^2 |\rho_v(t-s)|dt ds \leq \frac{4}{T} \int_0^\infty |\rho_v(r)| dr \big(Z_T^v+\rho_v(0)\big) \\[1.5 ex]
&\leq \frac{4}{T} \mathfrak{p}_0\mathfrak{K}_\infty \int_0^\infty \exp(-\mathfrak{r}_0 r) dr \big(Z_T^v+\mathfrak{K}_\infty\big)=\frac{4}{T} \frac{\mathfrak{p}_0\mathfrak{K}_\infty}{\mathfrak{r}_0}(Z_T^v+\mathfrak{K}_\infty).
\end{align*}
Consequently, the conditions of Theorem  \ref{Mallbound} are satisfied with $\al = \frac{4}{T} \frac{\mathfrak{p}_0\mathfrak{K}_\infty}{\mathfrak{r}_0}$ 
and $\be= \frac{4}{T} \frac{\mathfrak{p}_0\mathfrak{K}_\infty^2}{\mathfrak{r}_0}$, which completes the proof of Proposition \ref{PropH4} since
$\P(|Z_T^v|\geq x)= \P(Z_T^v\geq x) + \P(Z_T^v\leq -x)$.
\end{proof}

The statement of Proposition \ref{PropH4} corresponds to assumption (H4) in \cite{GM19}, which has been shown to be valid via  a log-Sobolev inequality only when $A_0$ is \textit{symmetric} (cf. \cite[Theorem]{GM19}). In other words, the extra assumption (H4) is not required as it directly follows from the modelling setup. 

The next theorem proves the restricted eigenvalue property.

\begin{theo} \label{RestrEigen}
Suppose that assumption (H) is satisfied and define $\mathfrak{k}_\infty:=\lambda_{\min}(C_\infty)>0$. Then for any $\epsilon_0 \in (0,1)$ it holds that 
\bee	
\P\Big(  \inf_{V \in \ca(s,c_0)} \frac{\|VX\|_{L^2}^2}{\|V\|_2^2}\geq \frac{\mathfrak{k}_\infty}{2} \Big)\geq 1-\epsilon_0,  
\eee
for all 
\bee \label{Tconcineq2}
T\geq  T_0(\epsilon_0,s,c_0):=\mathfrak{T}_0(\epsilon_0,s,c_0)
 \Big( (4s+1)\log{d}-2s\big( \log{\frac{2s}{21}}-1 \big)+\log{\frac{2}{\epsilon_0}} \Big), 
\eee
where  the constant $\mathfrak{T}_0(\epsilon_0,s,c_0)$ is defined as
$$
\mathfrak{T}_0(\epsilon_0,s,c_0)= \frac{144\mathfrak{p}_0\mathfrak{K}_\infty(c_0+2)^2(\mathfrak{k}_\infty+18(c_0+2)^2\mathfrak{K}_\infty)}{\mathfrak{r}_0 \mathfrak{k}_\infty^2}. 
$$
\end{theo}

\begin{proof} See Section \ref{sec6.2}.
\end{proof}

The next corollary presents a deviation bound for the quantity $\wC_T$. 

\begin{cor} \label{cor1}
For any $\epsilon_0>0$ and $T\geq  T_0(\epsilon_0,s,c_0)$ it holds that 
\bee	
\P\left(  \inf_{V \in \mathcal{C}(s,c_0)} \frac{\|VX\|_{L^2}^2}{\|V\|_{2}^2}\geq \frac{\mathfrak{k}_\infty}{2},~  \|\text{\rm diag } \wC_T\|_\infty \leq  \mathfrak{m}_\infty+ \frac{\mathfrak{k}_\infty}{2} ,~
\| \wC_T\|_\infty  \leq \mathfrak{M}_\infty +\frac{3\mathfrak{k}_\infty}{2}
\right)\geq 1-\epsilon_0,  
\eee
where $\mathfrak{m}_{\infty} := \|\text{\rm diag }C_{\infty}\|_{\infty}$ and $\mathfrak{M}_{\infty} := \|C_\infty\|_{\infty}$.
\end{cor}

\begin{proof} 
See Section \ref{sec6.3}.
\end{proof}

\subsection{Deviation bounds for the martingale term $\ep_T$ and final estimates} \label{sec3.2}
As mentioned earlier controlling the stochastic term $\lan \ep_T, V \ran_{\text{F}}$ for matrices $V \in \R^{d\times d}$ is crucial for the analysis 
of the estimators $\wA_{\text{L}}$ and  $\wA_{\text{D}}$. The martingale property of $\ep_T$ turns out to be the key in the next proposition. We remark
that the following result is an improvement of \cite[Theorem 8]{GM19}. 

\begin{prop} \label{MartProp} 
For any  $\epsilon_0 \in(0,1)$ the following inequality holds: 
\bee \label{ineqeps}
\P\left( \sup_{V \in \R^{d\times d}, V\not=0} \frac{\lan \ep_T,V \ran_{\text{\rm F}}}{\|V\|_1} \geq \mu \right) \leq \epsilon_0
\eee
for any 
\bee\label{Tconcineq}
T\geq  \frac{48\mathfrak{p}_0\mathfrak{K}_\infty}{\mathfrak{r}_0}\frac{\mathfrak{k}_\infty+6\mathfrak{K}_\infty}{\mathfrak{k}_\infty^2} \big( (2s+1)\ln{d}-s(\ln{s}-1)+\ln{(4/\epsilon_0)} \big) 
\eee
and 
\bee \label{muforeps}
\mu \geq \sqrt{\big(2\mathfrak{m}_\infty+\mathfrak{k}_\infty \big) \frac{\ln{(2d^2/\epsilon_0)}}{T}}.
\eee
\end{prop}

\begin{proof}
We first recall Bernstein's inequality for continuous local martingales. Let $(M_t)_{t\geq 0}$ be a real-valued continuous local martingale with 
quadratic variation $(\lan M\ran_t)_{t\geq 0}$. Then for any $a,b>0$ it holds that
\bee \label{Bern}
\P(M_t\geq a, \lan M\ran_t \leq b) \leq \exp(-a^2/(2b)). 
\eee
This result is a straightforward consequence of exponential martingale technique (cf. Chapter 4, Exercise 3.16 in \cite{RY05}).

By definition $\ep_T^{ij}=\frac{1}{T}\int_0^TdW_t^i X_t^j$ is a continuous martingale with $\lan \ep^{ij} \ran_T=\frac{1}{T}\wC_T^{ii}$. Therefore, we obtain
by   Corollary \ref{cor1} and \eqref{Bern}
\begin{gather*}
\P\left( \sup_{V \in \R^{d\times d}, V\not=0} \frac{\lan \ep_T,V \ran_{\text{\rm F}}}{\|V\|_1} \geq\mu \right) \leq \P\Big(  \|\text{diag }\wC_T\|_\infty > \mathfrak{m}_\infty+ \frac{\mathfrak{k}_\infty}{2}\Big) \\
+ \P\left( \sup_{V \in \R^{d\times d}, V\not=0} \frac{\lan \ep_T,V \ran_{\text{\rm F}}}{\|V\|_1} \geq\mu,~ \|\text{diag }\wC_T\|_\infty \leq \mathfrak{m}_\infty+ \frac{\mathfrak{k}_\infty}{2}\right) \\
 \leq \sum_{i,j=1}^d \P \Big( \ep_T^{ij} \geq \mu,~ \langle \ep^{ij} \rangle_T \leq \frac{1}{T}\big( \mathfrak{m}_\infty+ \frac{\mathfrak{k}_\infty}{2}\big) \Big)  +\frac{\epsilon_0}{2}\\
 \leq d^2 \exp{\Big( -T \frac{\mu^2}{2 \mathfrak{m}_\infty+\mathfrak{k}_\infty} \Big) }+\frac{\epsilon_0}{2} \leq \epsilon_0,
\end{gather*}
which completes the proof. 
\end{proof}

Summarising all previous deviation bounds we obtain the following result.

\begin{cor} \label{cor2}
For  $s\geq s_0$ and $c_0>0$ define the event
\begin{gather*}
\mathcal{E}(s,c_0):=\Big\{\inf_{V \in \mathcal{C}(s,c_0)} \frac{\|VX\|_{L^2}^2}{\|V\|_{2}^2}\geq \frac{\mathfrak{k}_\infty}{2} \Big\} 
\bigcap \Big\{ \sup_{V \neq 0} \frac{\langle \ep_T, V \rangle_{\text{F}} }{\|V\|_1}\leq \frac{\lambda}{2} \Big\} \\
\bigcap \Big\{ \|\ep_T\|_\infty \leq \frac{\lambda}{2} \Big\}
\bigcap\Big\{  \|\wC_T\|_\infty \leq \mathfrak{M}_\infty +\frac{3\mathfrak{k}_\infty}{2} \Big\}.
\end{gather*}
Then, for any $\epsilon_0 \in (0,1)$, it holds that $\P(\mathcal{E}(s,c_0)) \geq 1- \epsilon_0$ for any $T \geq T_0\big( \epsilon_0/2,s,c_0 \big)$ 
and 
\bee
\lambda \geq 2 \sqrt{\big(2\mathfrak{m}_\infty+\mathfrak{k}_\infty \big) \frac{\ln{(2d^2/\epsilon_0)}}{T}}.
\eee
\end{cor}

\section{Oracle inequalities and error bounds for the Lasso and Dantzig estimators} \label{sec4}
\setcounter{equation}{0}
\renewcommand{\theequation}{\thesection.\arabic{equation}}

In this section we present the main theoretical results for the Lasso and Dantzig estimators. More specifically, we derive oracle inequalities 
for $\wA_{\text{L}}$ and $\wA_{\text{D}}$, and show the error bounds for the norms $\|\cdot\|_{L^2}$, $\|\cdot\|_{1}$ and $\|\cdot\|_{2}$. In particular, 
we establish the asymptotic equivalence between the  Lasso and Dantzig estimators. 

\subsection{Properties of the Lasso estimator} \label{sec4.1}
We start this subsection with proving a statement, which is important for obtaining oracle inequality for the Lasso estimator $\wA_{\text{L}}$.

\begin{lem} \label{lem2} 
Suppose that condition \eqref{constraint} holds. For any matrix $A \in \R^{d \times d}\setminus\{ 0 \}$ denote $\mathcal{A}:=\text{\rm supp}(A)$. Then for any $s\geq s_0$ and $c_0>0$ on $\mathcal{E}(s,c_0)$ the following inequality holds: 
\bee \label{IneqLemmaLasso}
 \|(\wA_{\text{L}}-A_0)X\|_{L^2}^2 + \la \|\wA_{\text{L}}-A\|_1 
\leq \|(A-A_0)X\|_{L^2}^2 + 4\la \|\wA_{\text{L}|\mathcal{A}}-A\|_1. 
\eee
In particular, it implies that $\wA_{\text{L}}-A_0 \in \mathcal{C}(s_0,3)$ on $\mathcal{E}(s,c_0)$. 
\end{lem}

\begin{proof}
Let us set  $\de_L(A):=A-\wA_{\text{L}}$. Applying Lemma \ref{lem1} we obtain the following inequality 
\begin{gather*}
 \|(\wA_{\text{L}}-A_0)X\|_{L^2}^2+ \la \|\de_L(A) \|_1 \\[1.0 ex]
 \leq   \|(A-A_0)X\|_{L^2}^2+2 \lan \ep_T, \de_L(A)\ran_{\text{F}}  +\la \|\de_L(A) \|_1 + 2\la \big( \|A\|_1 - \|\wA_{\text{L}}\|_1 \big). 
 \end{gather*}
 Hence, on $\mathcal{E}(s,c_0)$ it holds that 
 \begin{gather*}
 \|(\wA_{\text{L}}-A_0)X\|_{L^2}^2+ \la \|\de_L(A) \|_1
 \leq   \|(A-A_0)X\|_{L^2}^2+ 2\la \big(\|\de_L(A) \|_1+ \|A\|_1 - \|\wA_{\text{L}}\|_1 \big). 
 \end{gather*}
 We observe next that $\|\de_L(A) \|_1+ \|A\|_1 - \|\wA_{\text{L}}\|_1  \leq 2 \|\de_L(A)_{|\mathcal A} \|_1$, which immediately implies \eqref{IneqLemmaLasso}. Applying \eqref{IneqLemmaLasso} to $A=A_0$ we deduce that 
$$
\|\de_L(A_0)\|_1 \leq 4 \|\de_L(A_0)_{|\mathcal{A}}\|_1 \leq 4\|\de_L(A_0)_{| \mathcal{I}_{s_0}(\de_L(A_0))}\|_1, 
$$
where the last inequality holds due to the sparsity assumption $\|A_0\|_0\leq s_0$. Consequently, $\wA_{\text{L}}-A_0 \in \mathcal{C}(s_0,3)$
and the proof is complete. 
\end{proof}

We are now in the position to present an oracle inequality for the Lasso estimator $\wA_{\text{L}}$, which is one of the main results of our paper. 

\begin{theo} \label{Lassothm}
Fix $\gamma>0$ and $\epsilon_0 \in (0,1)$. Consider the Lasso estimator $\wA_{\text{L}}$ defined at \eqref{lasso} and assume that condition (H) holds. Then for
\bee
\la \geq 2 \sqrt{\big(2\mathfrak{m}_\infty+\mathfrak{k}_\infty \big) \frac{\ln{(2d^2/\epsilon_0)}}{T}}
\eee
and $T\geq T_0\big( \epsilon_0/2, s_0,3+4/ \ga \big)$, with probability at least $1-\epsilon_0$ it holds that
\bee
\|(\wA_{\text{L}}-A_0)X\|_{L^2}^2  \leq (1+\gamma) \inf_{A:~ \|A\|_0\leq s_0}\Big\{ \|(A-A_0)X\|_{L^2}^2 +  \frac{9(2+\gamma)^2}{2\mathfrak{k}_\infty\gamma(1+\gamma)} \|A\|_0 \lambda^2\Big\}. 
\eee
\end{theo}

\begin{proof}
Consider an arbitrary matrix $A\in \R^{d\times d}$ with $\|A\|_0\leq s_0$. Then, on  $\mathcal{E}(s_0,3+4/ \gamma)$, according to Lemma~\ref{lem2}
and Cauchy-Schwarz inequality:
\begin{align} \label{usein}
\|(\wA_{\text{L}}-A_0)X\|_{L^2}^2 + \la \|\wA_{\text{L}}-A\|_1 & \leq \|(A-A_0)X\|_{L^2}^2 + 4\la  \|\wA_{\text{L}|\mathcal{A}}-A\|_1 \nonumber\\
&\leq \|(A-A_0)X\|_{L^2}^2 + 4\la \sqrt{\|A\|_0} \|\wA_{\text{L}|\mathcal{A}}-A\|_2. 
\end{align}
Now, if $4\la  \|\wA_{\text{L}|\mathcal{A}}-A\|_1 \leq \ga \|(A-A_0)X\|_{L^2}^2$ the result immediately follows from  Lemma~\ref{lem2}. Hence, we only
need to treat the case $4\la  \|\wA_{\text{L}|\mathcal{A}}-A\|_1 > \ga \|(A-A_0)X\|_{L^2}^2$. The latter implies that  
$\wA_{\text{L}}-A_0 \in \ca(s_0,3+4/\ga)$ due to  \eqref{usein}. Then, on the event $\mathcal{E}(s_0,3+4/ \gamma)$, we have
\begin{gather*}
\|\wA_{\text{L}|\mathcal{A}}-A\|_2^2 \leq \|\wA_{\text{L}}-A\|_2^2 \leq \frac{2}{\mathfrak{k}_\infty} \|(\wA_{\text{L}}-A)X\|_{L^2}^2 
\end{gather*}
and consequently we obtain from \eqref{usein} that 
\begin{gather*}
\|(\wA_{\text{L}}-A_0)X\|_{L^2}^2
\leq 
\|(A-A_0)X\|_{L^2}^2 +  3\lambda \sqrt{\frac{2\|A\|_0}{\mathfrak{k}_\infty}}\|(\wA_{\text{L}}-A)X\|_{L^2}^2\\
\leq \|(A-A_0)X\|_{L^2}^2 +  3\lambda \sqrt{\frac{2\|A\|_0}{\mathfrak{k}_\infty}}\big( \|(\wA_{\text{L}}-A_0)X\|_{L^2}^2 + \|(A-A_0)X\|_{L^2}^2 \big).
\end{gather*}
Using the inequality $2xy\leq ax^2 +y^2/a$ for $a>0$, we then conclude that 
$$
\|(\wA_{\text{L}}-A_0)X\|_{L^2}^2
\leq (1+\gamma)\|(A-A_0)X\|_{L^2}^2+\frac{9(2+\gamma)^2}{2\mathfrak{k}_\infty\gamma(1+\gamma)} \|A\|_0 \lambda^2,  
$$
which completes the proof. 
\end{proof}

Theorem \ref{Lassothm} enables us to find upper bounds on the various norms of $\wA_{\text{L}}-A_0$ as well as on the sparsity of $\wA_{\text{L}}$.
We remark that the bound in \eqref{LassoineqM} will be useful to provide the connection between the Lasso and Dantzig estimators in the next subsection.  

\begin{cor} \label{corLasso}
Fix $\epsilon_0 \in (0,1)$. Consider the Lasso estimator $\wA_{\text{L}}$ defined in \eqref{lasso} and assume that conditions 
\eqref{constraint} and (H) hold. Then for
\bee
\la \geq 2 \sqrt{\big(2\mathfrak{m}_\infty+\mathfrak{k}_\infty \big) \frac{\ln{(2d^2/\epsilon_0)}}{T}}
\eee
and $T\geq T_0\big(\epsilon_0/2,s_0,3 \big)$, with probability at least $1-\epsilon_0$, it holds that
\begin{align} \label{LassoineqL2}
\|(\wA_{\text{L}}-A_0)X\|_{L^2}^2 & \leq \frac{18}{\mathfrak{k}_\infty } s_0 \lambda^2 \\[1.0 ex]
\label{LassoineqF}
\|\wA_{\text{L}}-A_0\|_{2}^2 &\leq \frac{36}{\mathfrak{k}_\infty^2 } s_0 \lambda^2 \\[1.0 ex]
\label{Lassoineq1}
\|\wA_{\text{L}}-A_0\|_{1} &\leq \frac{24}{\mathfrak{k}_\infty} s_0 \lambda \\[1.0 ex]
\label{LassoineqM}
\|\wA_{\text{L}}\|_{0} &\leq \Big( 48 \frac{\mathfrak{M}_\infty}{\mathfrak{k}_\infty}+72 \Big) s_0. 
\end{align}
\end{cor} 

\begin{proof}
On the event $\mathcal{E}(s_0,3)$, taking $A=A_0$ and $\mathcal{A}_0=\text{supp}(A_0)$, we obtain the inequality 
\begin{gather*}
\|(\wA_{\text{L}}-A_0)X\|_{L^2}^2 + \la \|\wA_{\text{L}}-A\|_1
\leq 4\la \|\wA_{\text{L}|\mathcal{A}}-A\|_1
\end{gather*}
due to Lemma \ref{lem2}. Since on $\mathcal{E}(s_0,3)$ we have $\wA_{\text{L}}-A_0\in\mathcal{C}(s_0,3)$, we conclude that 
\begin{gather*}
\|(\wA_{\text{L}}-A_0)X\|_{L^2}^2 \leq 
3 \la \|\wA_{\text{L}|\mathcal{A}}-A\|_1 \\
 \leq 3 \lambda \sqrt{s_0} \|\wA_{\text{L}|\mathcal{A}}-A\|_2 \leq 3\lambda \sqrt{\frac{2s_0}{\mathfrak{k}_\infty}}\|(\wA_{\text{L}}-A_0)X\|_{L^2}^2
.\end{gather*}
This gives \eqref{LassoineqL2} and \eqref{LassoineqF}. Moreover, on the same event it holds 
$$
\|\wA_{\text{L}|\mathcal{A}}-A\|_1 \leq 4 \sqrt{s_0} \|\wA_{\text{L}|\mathcal{A}}-A\|_2
$$
and hence \eqref{Lassoineq1} follows. 

Now, it remains to prove \eqref{LassoineqM}. Note that necessary and sufficient condition for $\wA_{\text{L}}$ to be the solution of the optimisation problem \eqref{lasso} is the existence of a matrix $B \in \partial \|\wA_{\text{L}}\|_1$ such that 
$$
\ep_T + \big( \wA_{\text{L}} - A_0 \big) \wC_T +\lambda B = 0.  
$$
Furthermore, $\wA_{\text{L}}^{ij}\neq 0$ implies that $B^{ij}=\text{sign} (\wA_{\text{L}}^{ij}) $. Thus, we conclude that 
\begin{gather*}
\|(\wA_{\text{L}}-A_0)\wC_T\|_{1}=\|\lambda B +\ep_T\|_1 = \sum_{i,j=1}^d \left|\lambda B^{ij}+\ep_T^{ij} \right| \\
\geq \sum_{i,j: \wA_{\text{L}}^{ij}\neq 0} \left |\lambda B^{ij}+\ep_T^{ij}\right| 
\geq \sum_{i,j: \wA_{\text{L}}^{ij}\neq 0} \left |\lambda -|\ep_T^{ij}| \right|
\geq \|\wA_{\text{L}}\|_{0}\frac{\lambda}{2}, 
\end{gather*}
where the last inequality holds on $\mathcal{E}(s_0,3)$. On the other hand, on the same event we obtain
\begin{gather*}
\|(\wA_{\text{L}}-A_0)\wC_T\|_{1} \leq \|\wC_T\|_\infty \|\wA_{\text{L}}-A_0\|_{1} \leq \left( \mathfrak{M}_\infty + \frac{3 \mathfrak{k}_\infty}{2}\right)\frac{24}{\mathfrak{k}_\infty} s_0 \lambda, 
\end{gather*}
which implies \eqref{LassoineqM}. 
\end{proof}

The upper bounds in  \eqref{LassoineqL2}-\eqref{Lassoineq1} improve the bounds obtained in \cite[Corollary 1]{GM19} and they are in line with the classical results for linear regression models. We recall that the paper \cite{GM19} considers row sparsity of the unknown parameter 
$A_0$, i.e.
\[
\|A_0^{i}\|_0 \leq \mathfrak{s} \qquad \text{for all } 1\leq i \leq d,
\]
where $A_0^{i}$ denotes the $i$th row of $A_0$. Obviously, this constraint corresponds to $s_0=d \mathfrak{s}$ in our setting. The authors 
of  \cite{GM19} obtained the upper bound for  $\|\wA_{\text{L}}-A_0\|_{2}^2$ of order
\[
\frac{d \mathfrak{s} (\log d+\log\log T)}{T}
\]
in contrast to our improved bound $T^{-1}d \mathfrak{s} \log d$. Thus, we essentially match the lower bound
\[
\inf_{\wA} \sup_{A:~\max_i\|A_0^{i}\|_0 \leq \mathfrak{s}} \E[\|\wA-A\|_{2}^2] \geq \frac{c_1  d \mathfrak{s} \log (c_2 d/\mathfrak{s})}{T}  
\qquad \text{for some } c_1,c_2>0,
\]
which has been derived in \cite[Theorem 2]{GM19}.

The authors of \cite{GM19} have introduced the adaptive Lasso estimator, which is defined as   
\[
\wA_{\text{ad}}:= \argmin_{A\in \R^{d\times d}} \left( \lll_T(A) + \la \|A\circ |\wA_{\text{ML}}|^{-\gamma} \|_1\right),
\] 
where $\circ$ denotes the Hadamard product and $(|\wA_{\text{ML}}|^{-\gamma})_{ij}:=|\wA_{\text{ML}}^{ij}|^{-\ga}$ for a $\ga>0$. They have proved that the adaptive estimator $\wA_{\text{ad}}$  is consistent for support selection and showed the asymptotic normality of $\wA_{\text{ad}}$ when restricted to the elements in $\text{supp}(A_0)$; see \cite[Theorem 4]{GM19}.

\subsection{Properties of the Dantzig estimator} \label{sec4.2}

In this subsection we will establish a connection between the prediction errors associated with the Lasso and Dantzig estimators. This step is essential for the derivation of error bounds for $\wA_{\text{D}}$. Our results are an extension of the study in \cite{BRT09}, where it was shown that under sparsity conditions, the Lasso and the Dantizg estimators show similar behaviour for linear regression and for nonparametric regression models, for $l_2$ prediction loss and for $l_p$ loss in the coefficients for $1 \leq p \leq 2.$ 

In what follows, we will derive analogous bounds for the Ornstein-Uhlenbeck process.
\begin{prop} \label{propDantzig}
Consider the Lasso estimator $\wA_{\text{D}}$ defined in \eqref{Dantzig}  and assume that condition (H) holds. 
\begin{itemize}
\item[(i)] Define $\delta_D(A):=A-\wA_{\text{D}}$ and $\mathcal{A}:=\text{supp}(A)$, and assume that $A$ satisfies the Dantzig constraint \eqref{Lassoconstr}. Then it holds that 
$$
\|\delta_D(A)_{|\mathcal{A}^c}\|_1 \leq \|\delta_D(A)_{|\mathcal{A}}\|_1 .
$$
\item[(ii)]  On the event $\big\{ \|\wA_{\text{L}}\|_0 \leq s \big\} \cap \mathcal{E}(s,1)$ the following inequality holds: 
\bee
\Big| \|(\wA_{\text{L}}-A_0)X\|_{L^2}^2 -\|(\wA_{\text{D}}-A_0)X\|_{L^2}^2 \Big| \leq  \frac{18 }{\mathfrak{k}_\infty} \|\wA_{\text{L}}\|_0\lambda^2. 
\eee
\end{itemize}
\end{prop} 

\begin{proof} See Section \ref{sec6.4}.
\end{proof}

Proposition \ref{propDantzig} implies an oracle inequality for the Dantzig estimator, which is formulated in the next theorem.  

\begin{theo} \label{Dantzigthm}
Fix $\gamma>0$ and $\epsilon_0 \in (0,1)$. Consider the Dantzig estimator $\wA_{\text{D}}$ defined in~\eqref{Dantzig} and assume that conditions
\eqref{constraint} and   (H) hold. Then for
\bee
\lambda \geq 2 \sqrt{\big(2\mathfrak{m}_\infty+\mathfrak{k}_\infty \big) \frac{\ln{(2d^2/\epsilon_0)}}{T}}
\eee
and $T\geq T_0\big( \epsilon_0/2, ( 48 \frac{\mathfrak{M}_\infty}{\mathfrak{k}_\infty}+72)s_0,3+4/ \ga \big)$, with probability at least $1-\epsilon_0$, it holds that
\bee \label{dantzigoracle}
\|(\wA_{\text{D}}-A_0)X\|_{L^2}^2\leq (1+\gamma) \inf_{A: \|A\|_0 = s_0} \left\{ \|(A-A_0)X\|_{L^2}^2 + C_D(\gamma) s_0 \lambda^2 \right\},
\eee
where 
$$
C_D(\gamma)= \frac{18}{\mathfrak{k}_\infty} \Big(\frac{(\gamma+2)^2}{4 \gamma} + 48 \frac{\mathfrak{M}_\infty}{ \mathfrak{k}_\infty}+72 \Big) .
$$
\end{theo}

\begin{proof}
Consider matrix $A\in \R^{d\times d}$ such that $\|A\|_0=s_0$. Then, on the event $\mathcal{E}(s_1,3+4/\gamma)$, according to Proposition \eqref{propDantzig}
\begin{gather*}
 \|(\wA_{\text{D}}-A_0)X\|_{L^2}^2
 \leq \|(\wA_{\text{L}}-A_0)X\|_{L^2}^2  + \frac{18 }{\mathfrak{k}_\infty} \Big( 48 \frac{\mathfrak{M}_\infty}{\mathfrak{k}_\infty}+72 \Big) s_0 \lambda^2. 
\end{gather*}
On the other hand, due to Theorem \ref{Lassothm}, we deduce that 
\bee
\|(\wA_{\text{L}}-A_0)X\|_{L^2}^2
\leq (1+\gamma)\|(A-A_0)X\|_{L^2}^2+ \frac{9(\gamma+2)^2}{2\mathfrak{k}_\infty \gamma} s_0 \lambda^2. 
\eee
Combining both  inequalities yields \eqref{dantzigoracle}.
\end{proof}

The statements of Theorems   \ref{Lassothm}  and \ref{Dantzigthm} suggest that the Lasso and Dantzig estimators are asymptotically equivalent. This is in line with the theoretical findings in linear regression models as it has been shown in \cite{BRT09}. More specifically, we obtain the following result, which is a direct analogue of  Corollary  \ref{corLasso}. 

\begin{cor} \label{corDantzig} 
Fix $\epsilon_0 \in (0,1)$. Consider the Dantzig estimator $\wA_{\text{D}}$ defined in \eqref{Dantzig} and assume that conditions \eqref{constraint} and  (H) hold. Then for
\bee
\la \geq 2 \sqrt{\big(2\mathfrak{m}_\infty+\mathfrak{k}_\infty \big) \frac{\ln{(2d^2/\epsilon_0)}}{T}}
\eee
and $T\geq T_0\big(\epsilon_0/2,s_0,1 \big)$, with probability at least $1-\epsilon_0$, it holds that
\begin{align} \label{DantzigineqL2}
\|(\wA_{\text{D}}-A_0)X\|_{L^2}^2 & \leq \frac{18}{\mathfrak{k}_\infty } s_0 \lambda^2 \\[1.0 ex]
\label{DantzigineqF}
\|\wA_{\text{D}}-A_0\|_{2}^2 &\leq \frac{36}{\mathfrak{k}_\infty^2 } s_0 \lambda^2 \\[1.0 ex]
\label{Dantzigineq1}
\|\wA_{\text{D}}-A_0\|_{1} &\leq \frac{24}{\mathfrak{k}_\infty} s_0 \lambda. 
\end{align}
\end{cor}

\begin{proof} 
Denote $\mathcal{A}_0=\text{supp}(A_0)$. On the event $\mathcal{E}(s_0,1)$ the matrix $A_0$ satisfies the Dantzig constraint \eqref{Lassoconstr}, 
$\wA_{\text{D}}-A_0 \in \mathcal{C}(s_0,1)$ and 
\begin{gather*}
\|(\wA_{\text{D}}-A_0)X\|_{L^2}^2 
 \leq \|(\wA_{\text{D}}-A_0)\wC_T\|_\infty \|\wA_{\text{D}}-A_0\|_1 \\
\leq 2 \big(  \|(\wA_{\text{D}}-A_0)\wC_T +\ep_T\|_\infty+ \|\ep_T\|_\infty \big)  \|\wA_{\text{D}|\mathcal{A}_0}-A_0\|_1 \\
 \leq 3 \lambda \sqrt{s_0} \|\wA_{\text{D}}-A_0\|_2 \leq 3\lambda \sqrt{\frac{2s_0}{\mathfrak{k}_\infty}}\|(\wA_{\text{D}}-A_0)X\|_{L^2}^2, 
\end{gather*}
which gives \eqref{DantzigineqL2} and \eqref{DantzigineqF}. Moreover, on the same event it holds that
$$
\|\wA_{\text{D}}-A_0\|_1 \leq 2 \sqrt{s_0} \|\wA_{\text{D}}-A_0\|_2,
$$
which completes the proof. 
\end{proof}

It is noteworthy to mention that even if in our case Lasso and Dantzig selector performances are equivalent, a potential strength of the Dantzig estimator over penalized likelihood methods such as Lasso is that it can be applied to settings in which no explicit
likelihoods or loss functions are available, and may be of interest in both computational and
theoretical context (see \cite{dicker2014} for more details).

\section{Numerical simulations} \label{sec5}
\setcounter{equation}{0}
\renewcommand{\theequation}{\thesection.\arabic{equation}}

This sections presents some numerical experiments on simulated data that illustrate our theoretical results. 

Our estimation methods are based on continuous observations of the the underlying process, which need to be discretised  for numerical simulations. We will use 500000 discretisation points over the time interval $[0,T]$ with $T=300$. Such approximation is sufficient for the illustration purpose, since further refinement of the grid does not lead to a significant improvement.  

In Figure \ref{Figure1} we demonstrate an example of the transition matrix $A_0\in \R^{15\times 15}$ and the corresponding maximum likelihood, Lasso and Dantzig estimators. Instead of giving numerical values of the entries of $A_0$ we use a colour code to highlight the sparsity. We observe
that MLE provides a good performance on the support, but it gives rather poor estimates outside the support. On the other hand, the superiority of the Lasso and Dantzig estimators, especially in terms of support recovery, is quite obvious even for  relatively small dimension of matrix.

Figure \ref{Figure2} demonstrates the relative error of the maximum likelihood, Lasso and Dantzig estimators compared to the norm of the true matrix.
We compute the relative error for dimensions $d=5,\ldots,20$ and for  $L^1$ and Frobenius norms. Figure \ref{Figure2} clearly shows
the improvement of performance of penalized estimation methods with growth of the dimension $d$ compared to the maximum likelihood estimation.
Indeed, we observe  that relative errors of maximum likelihood estimation grow linearly both in $L^1$ and Frobenius norms, while  relative errors of Lasso and Dantzig estimators decay in $d$. 
The sparsity of the true parameter $A_0$ was chosen equal to $s=0.3d^2$, which might explain the limiting behaviour of Lasso and Dantzig estimators
when $d$ is increasing.  Finally, we observe that relative errors for Lasso and Dantzig estimators are practically equivalent, which is exactly in accordance with our theoretical results.

\begin{figure} 
\centering
\subfloat[Transition matrix $A_0$]{
  \includegraphics[width=60mm]{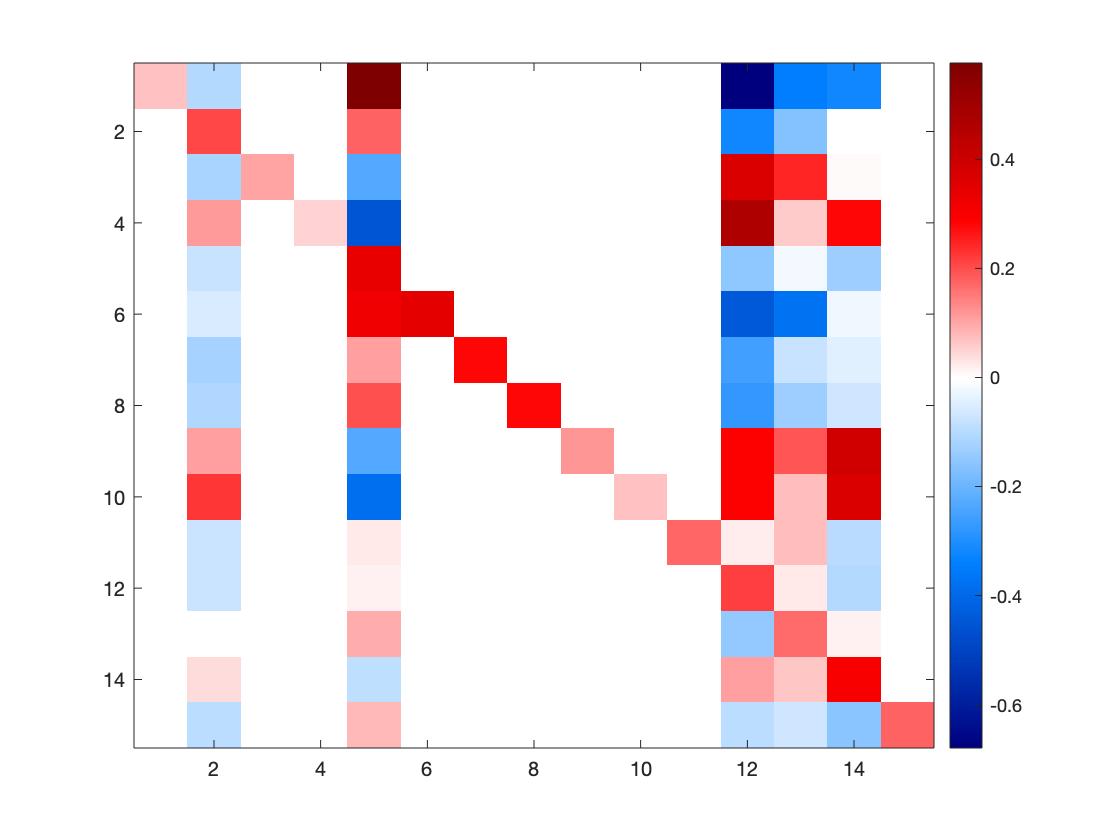}
}
\subfloat[MLE]{
  \includegraphics[width=60mm]{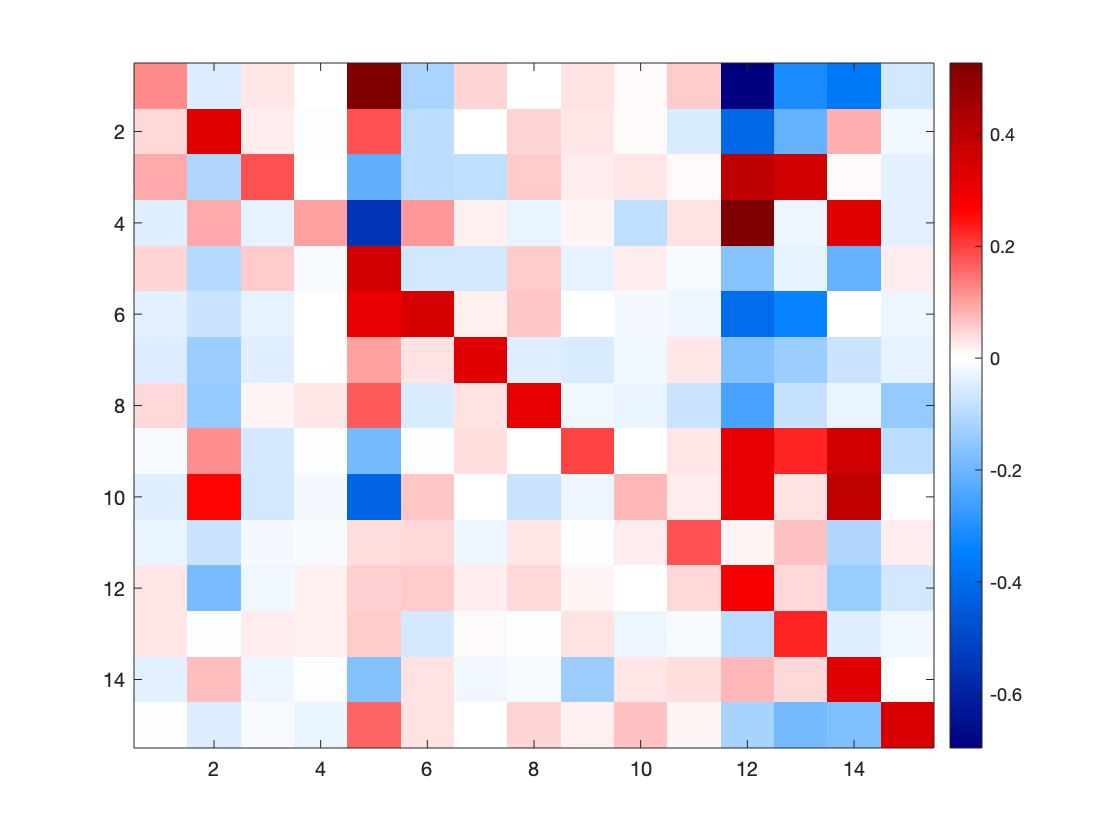}
}
\hspace{6pt}
\subfloat[Lasso]{
  \includegraphics[width=60mm]{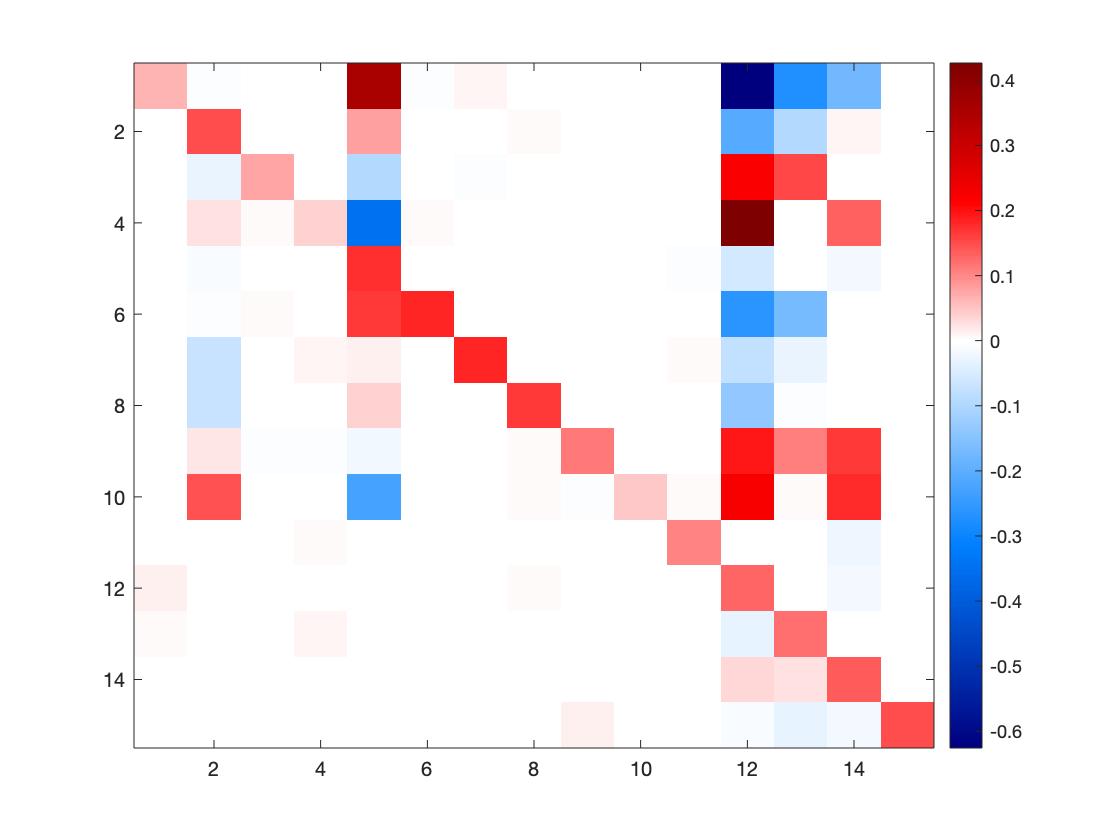}
}
\subfloat[Dantzig]{
  \includegraphics[width=60mm]{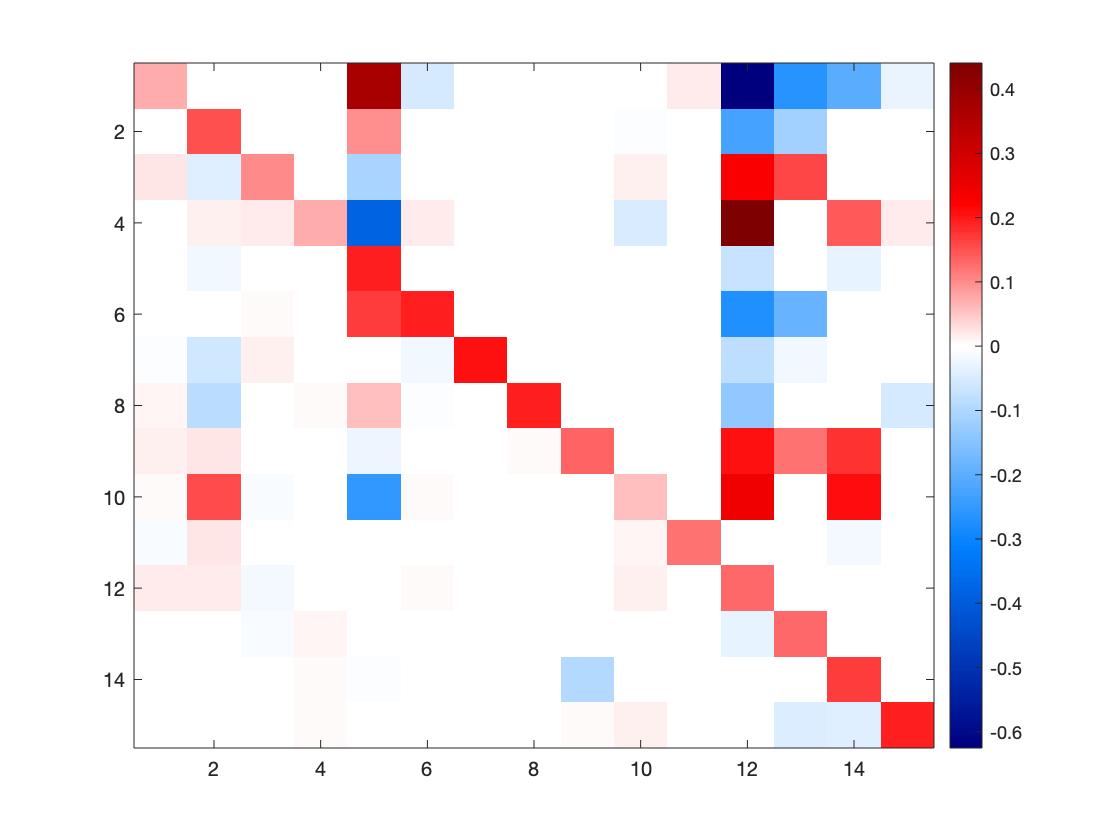}
}

\caption{Comparison of the true matrix with maximum likelihood, Lasso and Dantzig estimators.  }

 \label{Figure1}
 
\end{figure}

\begin{figure}
\centering
\subfloat[MLE - $L_1$-norm]{
  \includegraphics[width=65mm]{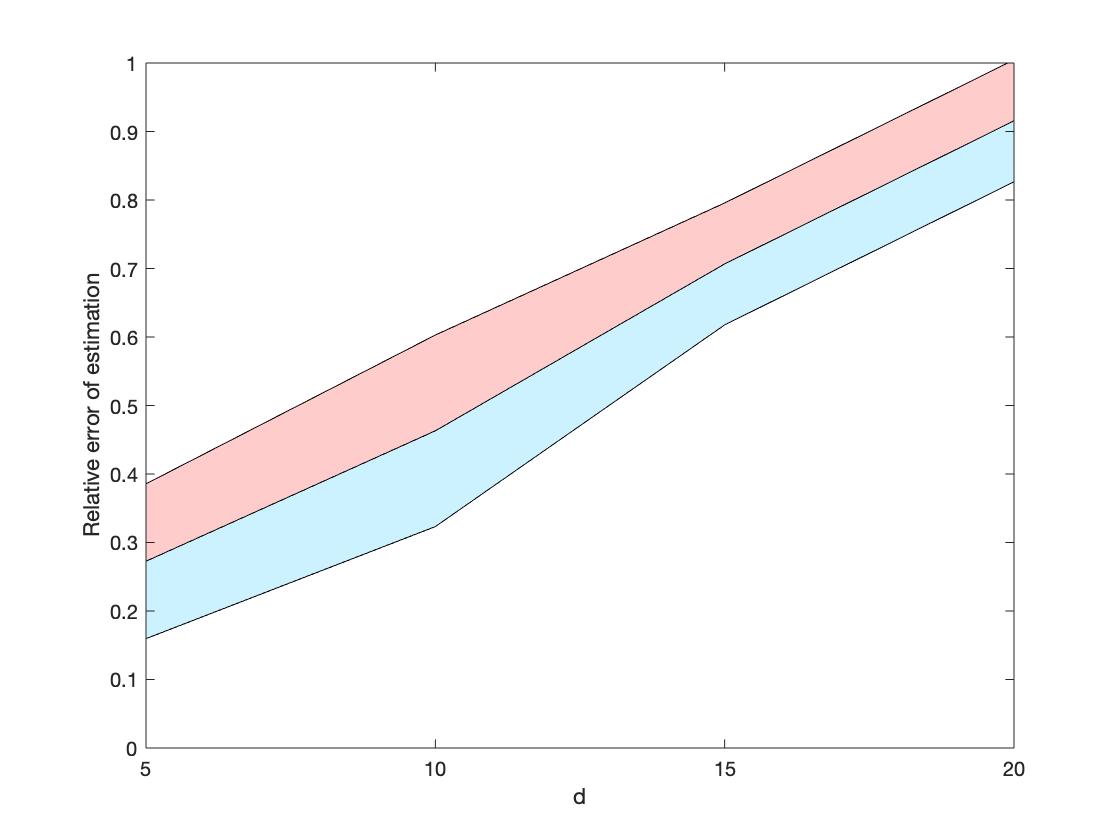}
}
\subfloat[MLE - Frobenius norm]{
  \includegraphics[width=65mm]{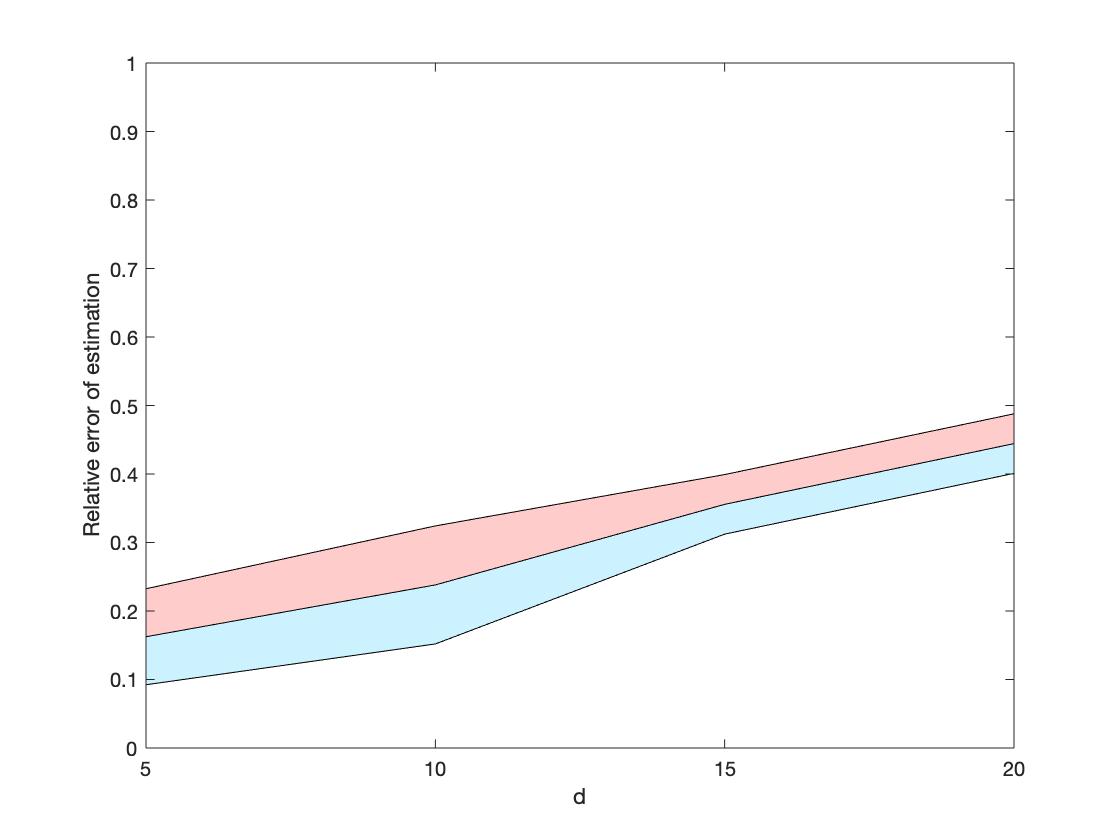}
}

\hspace{6pt}

\subfloat[Lasso - $L_1$-norm]{
  \includegraphics[width=65mm]{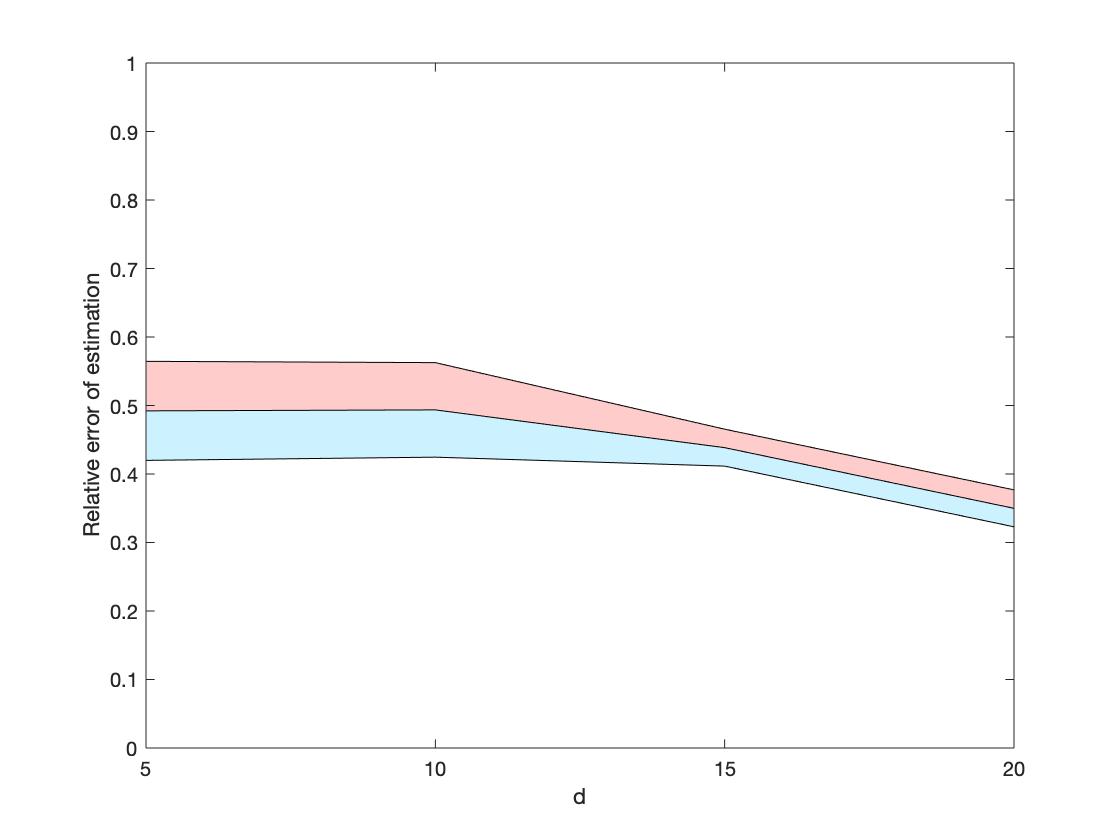}
}
\subfloat[Lasso - Frobenius norm]{
  \includegraphics[width=65mm]{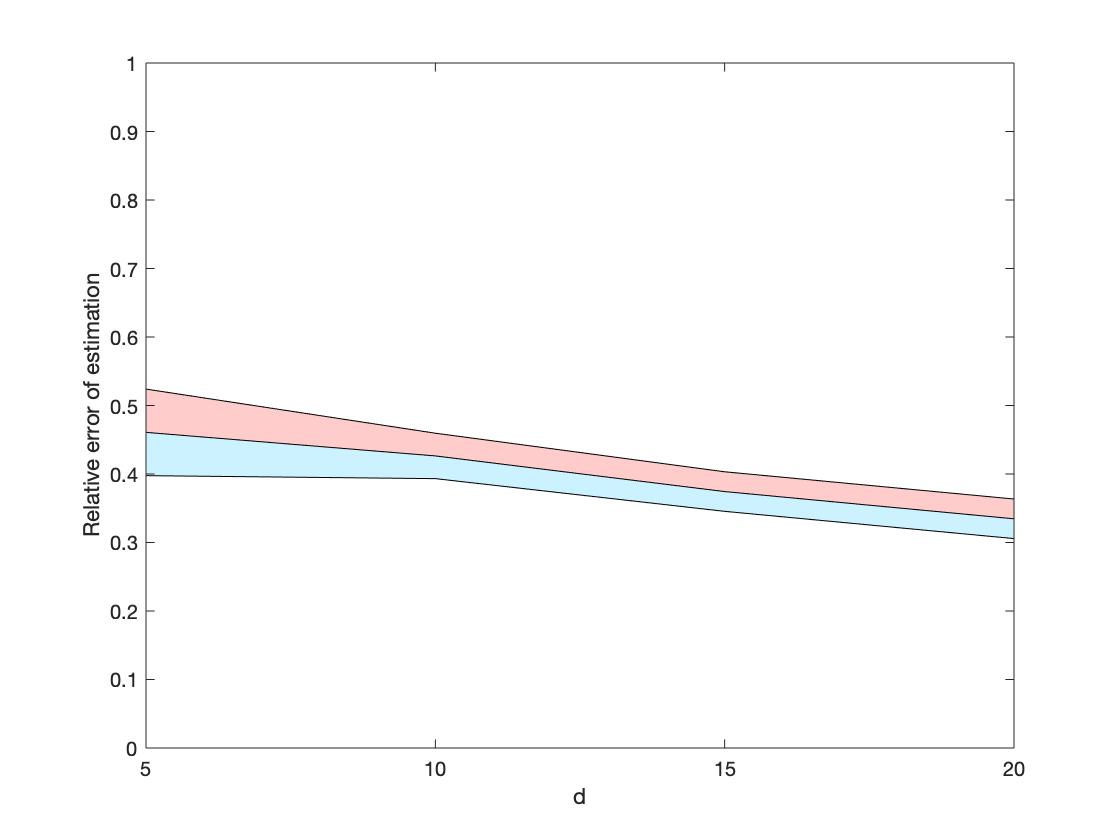}
}

\hspace{6pt}

\subfloat[Dantzig - $L_1$-norm]{
  \includegraphics[width=65mm]{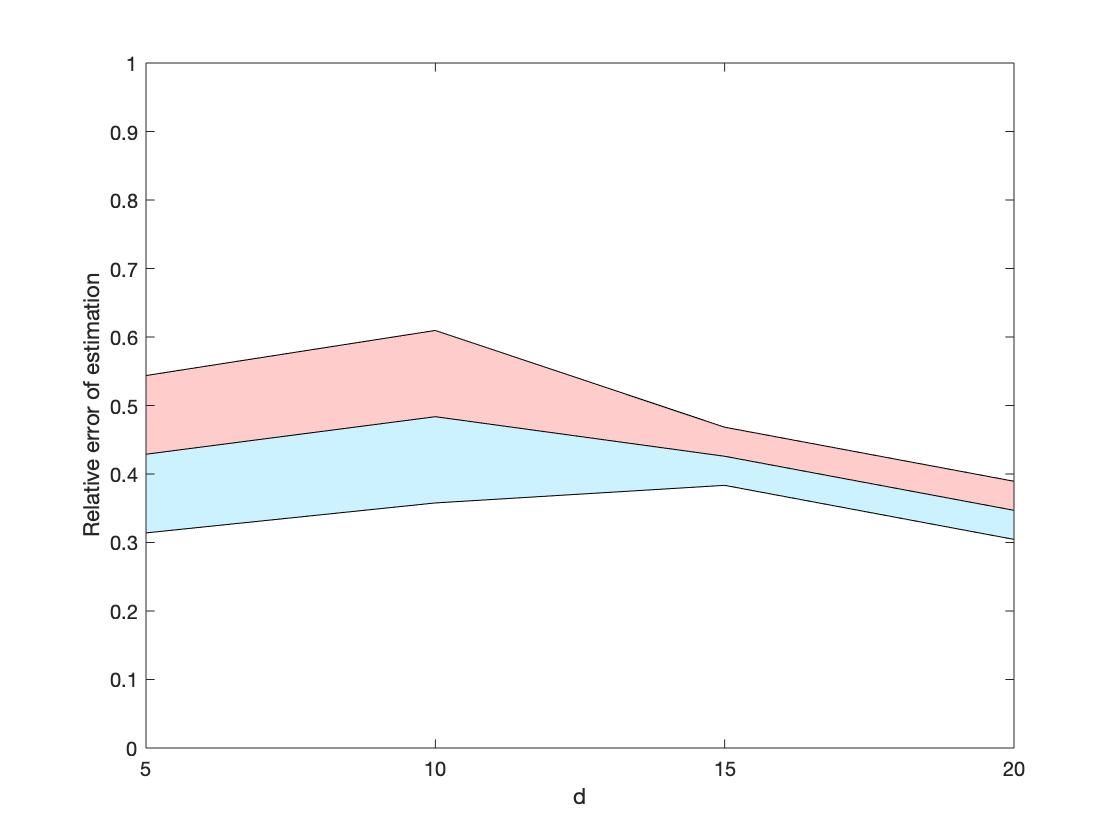}
}
\subfloat[Dantzig - Frobenius norm]{
  \includegraphics[width=65mm]{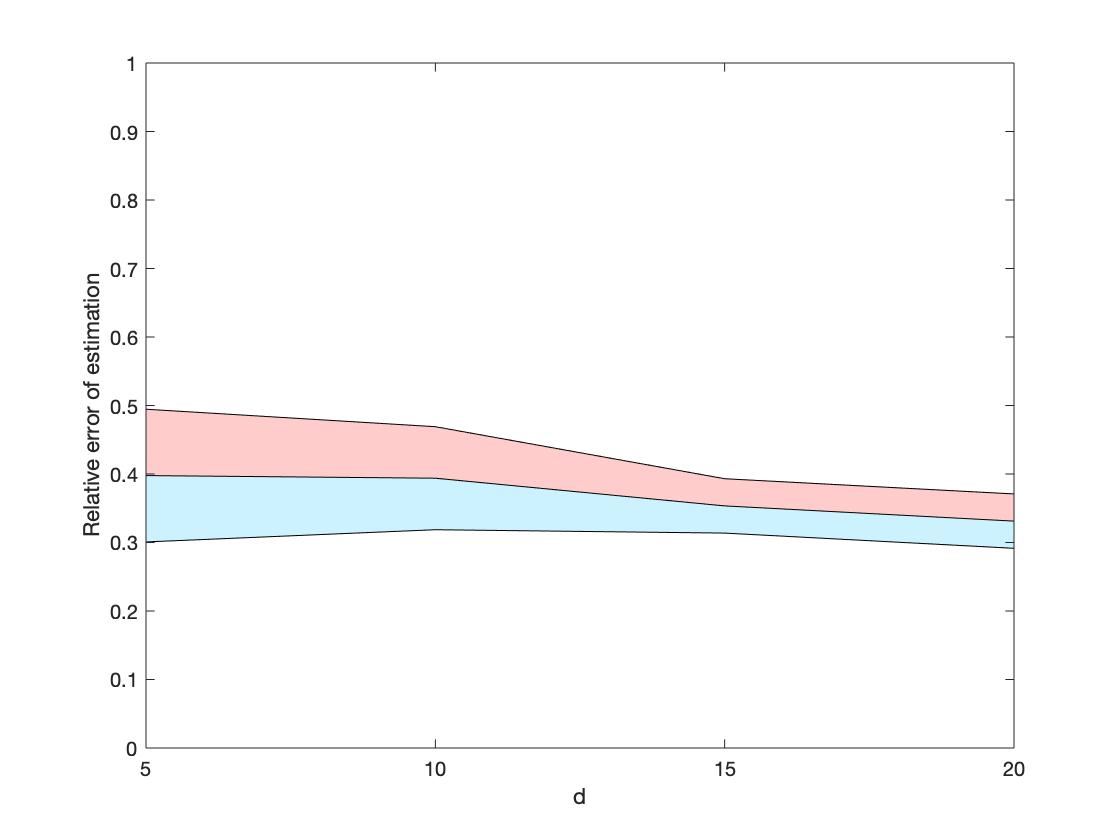}
}

\caption{Relative error of maximum likelihood, Lasso and Dantzig estimators in $L^1$  and Frobenius norms depending on $d$. Middle line corresponds to the mean and coloured areas correspond to the standard deviation of the error over 10 independent simulations. }

\label{Figure2}

\end{figure}

\section{Proofs} \label{sec6}
\setcounter{equation}{0}
\renewcommand{\theequation}{\thesection.\arabic{equation}}

\subsection{Proof of Theorem \ref{RestrEigen}} \label{sec6.2}
We first note the identity $\|VX\|_{L^2}^2=\text{tr}(V  \wC_T V^{\top})$. Replacing $\wC_T$ by its limit $C_{\infty}$ we deduce the inequality
$\text{tr}(V  C_{\infty} V^{\top}) \geq \mathfrak{k}_\infty>0$ and therefore 
\bee \label{decom}
\frac{\|VX\|_{L^2}^2}{\|V\|_{2}^2}=\frac{\text{tr}(V  C_{\infty} V^{\top})}{\|V\|_{2}^2} - \frac{\text{tr}(V  (C_{\infty}-\wC_T) V^{\top})}{\|V\|_{2}^2} \geq \mathfrak{k}_\infty - \frac{|\text{tr}(V  (C_{\infty}-\wC_T) V^{\top}) | }{\|V\|_{2}^2}. 
\eee
Next, we introduce the set $\ke(s):=\left\{ V \in \R^{d\times d} \setminus\{ 0 \} :~ \|V\|_0\leq s \right\}$. As is shown in Lemma \ref{lemSM1}  
it holds that 
\bee \label{mainineq}
\sup_{V \in \ca(s,c_0)} \frac{| \text{tr}(V (C_\infty - \wC_T) V^{\top}) | }{\|V\|_{2}^2}  \leq 3(c_0+2)^2 \sup_{V \in \mathcal{K}(2s)} 
\frac{| \text{tr}(V (C_\infty - \wC_T) V^{\top}) | }{\|V\|_{2}^2}.
\eee
Thus, it suffices to consider $\ke(s)$ instead of $\ca(s,c_0)$ in the following discussion. Observing \eqref{decom} we obtain that
\bee
\P\left(  \inf_{V \in \ke(s)} \frac{\|VX\|_{L^2}^2}{\|V\|_{2}^2} \geq \frac{\mathfrak{k}_\infty}{2} \right) 
\geq 
\P\left(  \sup_{V \in \ke(s)} \frac{| \text{tr}(V (C_\infty - \wC_T) V^{\top}) | }{\|V\|_{2}^2} \leq \frac{\mathfrak{k}_\infty}{2} \right).
\eee
For a matrix $V \in \ke(s)$ we denote its $j$-th row vector by $v^j$ and $\textbf{v}=\text{vec}(V)\in \R^{d^2}$. 
Moreover, we define a symmetric random matrix $\mathcal{D}_C=\text{id} \otimes (C_\infty-\wC_T)\in \R^{d^2 \times d^2}$. Then 
we deduce the identity
\bee \label{Vv}
\frac{\text{tr}(V (C_\infty - \wC_T) V^{\top})}{\|V\|_{2}^2}=\frac{\textbf{v}^{\top} \mathcal{D}_C \textbf{v}}{\|\textbf{v}\|_2^2}. 
\eee
According to Proposition \ref{PropH4} we obtain the following inequalities for any $x>0$:
\begin{gather*}
\P\left( \frac{|\textbf{v}^{\top} \mathcal{D}_C \textbf{v}|}{\|\textbf{v}\|_2^2} \geq x \right) \leq \P \left( \frac{\sum_{j=1}^d |v^j(C_\infty - \wC_T)(v^j)^{\top}|}{\sum_{j=1}^d \|v^j\|_2^2} \geq x \right) \\[1.5 ex]
\leq \sum_{j=1}^d \P \left( \frac{|v^j(C_\infty - \wC_T)(v^j)^{\top}|}{\|v^j\|_2^2} \geq x \right) \leq 2d \exp{(-T H_0(x))}. 
\end{gather*}
By Lemma \ref{lemSM2} we conclude that 
\bee
\P\left( \sup_{\textbf{v}\in\R^{d^2}\setminus\{0\}:~ \|\textbf{v}\|_0\leq s}\frac{|\textbf{v}^{\top} \mathcal{D}_C \textbf{v}|}{\|\textbf{v}\|_2^2} \geq 3x \right) \leq 2d \Big(\frac{21ed^2}{s}\Big)^s   \exp(-T H_0(x)) .
\eee
We deduce from  \eqref{Vv} that 
\bee
\P\left(  \inf_{V \in \ke(s)} \frac{\|VX\|_{L^2}^2}{\|V\|_{2}^2}\geq 3x \right) 
\geq 
1- 2d \Big(\frac{21ed^2}{s}\Big)^s   \exp{(-T H_0(x))}.
\eee
The latter statement together with \eqref{mainineq} implies the inequality 
\bee	
\P\Big(  \inf_{V \in \ca(s,c_0)} \frac{\|VX\|_{L^2}^2}{\|V\|_2^2}\geq \frac{\mathfrak{k}_\infty}{2} \Big)\geq 1-\epsilon_0,  
\eee
for all $T\geq  T_0(\epsilon_0,s,c_0)$, which completes the proof of Theorem \ref{RestrEigen}.

\subsection{Proof of Corollary \ref{cor1}} \label{sec6.3}
Let $e_{(i,j)}\in \R^{d\times d}$ be a matrix defined as $e_{(i,j)}^{kl}:=1_{(k,l)=(i,j)}$. We observe that
\begin{gather*}
  \left\{ \|\text{diag} (C_\infty -\wC_T)\|_\infty>\frac{\mathfrak{k}_\infty}{2} \right\} = \left\{\max_{1\leq j\leq d} \left|\text{tr}(e_{(j,j)} (C_\infty -\wC_T)
  e_{(j,j)}^{\top}\right| > \frac{\mathfrak{k}_\infty}{2} \right\}\\[1.0 ex]
  \subset \left\{  \sup_{V \in \mathcal{C}(s,c_0)} \frac{\left|\text{tr}(V (C_\infty -\wC_T)
  V^{\top}\right| }{\|V\|_{2}^2} > \frac{\mathfrak{k}_\infty}{2} \right\}.
\end{gather*}
Furthermore,  
\begin{gather*}
|C_\infty^{ij}-\wC_T^{ij}| = \left| \text{tr}(e_{(1,i)} (C_\infty - \wC_T) e_{(1,j)}^{\top}) \right|\\[1.0 ex] 
\leq \frac{1}{2} \left|\text{tr}((e_{(1,i)}
+e_{(1,j)} ) (C_\infty - \wC_T) (e_{(1,i)}+e_{(1,j)} )^{\top}) \right| \\[1.0 ex] 
+\frac{1}{2} \left| \text{tr}(e_{(1,i)} (C_\infty - \wC_T) e_{(1,i)}^{\top}) \right| +\frac{1}{2}\left| \text{tr}(e_{(1,j)} (C_\infty - \wC_T) e_{(1,j)}^{\top}) \right| \\[1.0 ex] 
\leq 3 \sup_{V \in \mathcal{C}(s,c_0)}\frac{\left|\text{tr}(V (C_\infty -\wC_T)V^{\top}\right| }{\|V\|_{2}^2}
\end{gather*}
and hence
\begin{gather*}
  \left\{ \|C_\infty -\wC_T\|_\infty> \frac{3 \mathfrak{k}_\infty}{2} \right\} = \left\{\max_{1\leq i,j\leq d}  \left| \text{tr}(e_{(1,i)} (C_\infty - \wC_T) e_{(1,j)}^{\top}) \right| > \frac{3 \mathfrak{k}_\infty}{2} \right\}\\[1.0 ex]
  \subset \left\{  \sup_{V \in \mathcal{C}(s,c_0)}\frac{\left|\text{tr}(V (C_\infty -\wC_T)V^{\top}\right| }{\|V\|_{2}^2} 
  > \frac{\mathfrak{k}_\infty}{2} \right\}.
\end{gather*}
This completes the proof of Corollary  \ref{cor1}.

\subsection{Proof of Proposition \ref{propDantzig}} \label{sec6.4}
Since $A$ satisfies the Dantzig constraint \eqref{Lassoconstr}, we deduce by definition of the Dantzig estimator:
\begin{gather*}
\|A\|_1 \geq \|\wA_{\text{D}}\|_1=\|A-\delta_D(A)_{|\mathcal{A}}\|_1+\|\delta_D(A)_{|\mathcal{A}^c}\|_1\\
\geq \|A\|_1 - \|\delta_D(A)_{|\mathcal{A}}\|_1 + \|\delta_D(A)_{|\mathcal{A}^c}\|_1,
\end{gather*}
which proves part (i). 

Now we show part (ii) of the proposition. Set $\de:=\wA_{\text{L}}-\wA_{\text{D}}$. Due to \eqref{LTformula} we deduce 
\bee \label{equivfirstineq}
\begin{gathered}
\|(\wA_{\text{L}}-A_0)X\|_{L^2}^2 -\|(\wA_{\text{D}}-A_0)X\|_{L^2}^2\\
= 2 \text{tr}\left(\left( \wA_{\text{D}} \wC_T +\ep_T-A_0\wC_T\right) \delta^{\top}\right) -2 \text{tr}\left(\ep_T \delta^{\top}\right)+ 
\text{tr}\left(\de \wC_T \delta^{\top}\right) \\
= 2 \text{tr}\left(\left( \wA_{\text{L}} \wC_T +\ep_T-A_0\wC_T\right) \delta^{\top}\right) -2 \text{tr}\left(\ep_T \delta^{\top}\right)-
\text{tr}\left(\de \wC_T \delta^{\top}\right).
\end{gathered}
\eee
The Dantzig constraint \eqref{Lassoconstr} implies the inequality
\bee
\left |\text{tr}\left(\left( \wA_{\text{D}} \wC_T +\ep_T-A_0\wC_T\right) \delta^{\top}\right)  \right| \leq \| \wA_{\text{D}} \wC_T +\ep_T-A_0\wC_T \|_\infty \|\delta\|_1\leq \lambda \|\delta\|_1,
\eee
and the same inequality holds for $\wA_{\text{D}}$ being replaced by $\wA_{\text{L}}$.
On $\mathcal{E}(s,1)$ we have
\bee
\left| \text{tr}\left(\ep_T \delta^{\top}\right) \right| \leq \frac{\lambda}{2} \|\delta\|_1. 
\eee
Furthermore, on $\big\{ \|\wA_{\text{L}}\|_0 \leq s \big\}$ it holds that $\delta \in \mathcal{C}(s,1)$ and we conclude from Theorem~\eqref{RestrEigen} that
\[
 \text{tr}\left(\de \wC_T \delta^{\top}\right)  \geq \frac{\mathfrak{k}_\infty}{2}\|\delta\|_2^2.
\]
We also have $ \|\delta\|_1 \leq 2  \|\delta_{|\text{supp}(\wA_{\text{L}})}\|_1 \leq 2 \|\wA_{\text{L}}\|_0^{1/2} \|\delta\|_2^2$. Observing the first 
identity of \eqref{equivfirstineq},  putting the previous estimates together and using the inequality $2xy\leq ax^2 + y^2/2$ for $a>0$, we obtain the following inequality 
\bee
 \|(\wA_{\text{D}}-A_0)X\|_{L^2}^2 - \|(\wA_{\text{L}}-A_0)X\|_{L^2}^2  \leq  \frac{18 }{\mathfrak{k}_\infty} \|\wA_{\text{L}}\|_0\lambda^2.
\eee
On the other hand, applying the second identity of \eqref{equivfirstineq}, we deduce that 
\bee
 \|(\wA_{\text{L}}-A_0)X\|_{L^2}^2 - \|(\wA_{\text{D}}-A_0)X\|_{L^2}^2  \leq  \frac{18 }{\mathfrak{k}_\infty} \|\wA_{\text{L}}\|_0\lambda^2,
\eee
which completes the proof.

\subsection{Some lemmas} \label{sec6.1}
In this subsection we present two results that can be easily deduced from Lemmas F.1, F.2 and F.3 from supplementary material of \cite{BM15}. We state their proofs for the sake of completeness. 

\begin{lem} \label{lemSM1}
It holds that 
\bee 
\sup_{V \in \ca(s,c_0)} \frac{| \text{\rm tr}(V (C_\infty - \wC_T) V^{\top}) | }{\|V\|_{2}^2}  \leq 3(c_0+2)^2 \sup_{V \in \mathcal{K}(2s)} 
\frac{| \text{\rm tr}(V (C_\infty - \wC_T) V^{\top}) | }{\|V\|_{2}^2}.
\eee
\end{lem}

\begin{proof}
First, recall the definition of the set $ \ca(s,c_0)$  in \eqref{Csc0} and  denote the unit balls by $\mathbb{B}_q(r):= \{ v \in \mathbb{R}^d: \|v\|_q \leq r\}$ for any $d \geq 1$ and $q \geq 0,\; r>0$. Furthermore, we introduce the notation $\mathcal{K}(s) = \mathbb{B}_0(s) \cap \mathbb{B}_2(1)$ for $ s\geq 1$. For any set $P$ we denote its closure and convex hull by $\text{cl}(P)$ and $\text{conv}(P),$ respectively.
By a direct application of Lemma F.1 from \cite{BM15}, we obtain the following approximation of cone sets by sparse sets: for any $S \subset \{1, \ldots, d\}$ with $|S|=s$ we get
\begin{align}\label{eq:p1} \ca(s,c_0) \cap \mathbb{B}_2 (1) \subseteq \mathbb{B}_1 \left((c_0 +1)\sqrt{s}\right) \cap \mathbb{B}_{2}(1)\subseteq (c_0 +2) \text{cl}(\text{conv} (\mathcal{K}(s))).
\end{align}
Next, by the statement of Lemma F.3 in \cite{BM15} we have that
\begin{equation}\label{eq:p2}
    \sup_{V \in \text{cl}(\text{conv} (\mathcal{K}(s)))}| \text{\rm tr}(V (C_\infty - \wC_T) V^{\top}) | \leq 3 \sup_{V \in \mathcal{K}(2s)}| \text{\rm tr}(V (C_\infty - \wC_T) V^{\top}) |.
\end{equation}
Thus, \eqref{eq:p1} combined with \eqref{eq:p2} yields the proof.
\end{proof}
\begin{lem} \label{lemSM2}
Let $\textbf{v}=\text{\rm vec}(V)\in \R^{d^2}$
and $\mathcal{D}_C=\text{\rm id} \otimes (C_\infty-\wC_T)\in \R^{d^2 \times d^2}$. 
Then it holds that 
\bee
\P\left( \sup_{\textbf{v}\in\R^{d^2}\setminus\{0\}:~ \|\textbf{v}\|_0\leq s}\frac{|\textbf{v}^{\top} \mathcal{D}_C \textbf{v}|}{\|\textbf{v}\|_2^2} \geq 3x \right) \leq 2d\Big(\frac{21ed^2}{s}\Big)^s   \exp(-T H_0(x)) ,
\eee
where the function $H_0$ has been introduced in Proposition \ref{PropH4}. 


\end{lem}
\begin{proof}
Choose $U \subset \{1,\ldots, d^2\}$ with $ |U|=s$, and define $$S_U = \left\{ \textbf{v} \in \mathbb{R}^{d^2}: \|\textbf{v}\|_2 \leq 1,\; \text{supp}(\textbf{v}) \subseteq U\right\}. $$ Then $\mathcal{K}(s) = \bigcup_{|U|\leq s} S_U.$ In what follows, we choose $\mathcal{A} = \{u_1, \ldots, u_m\},$ which is a $\frac{1}{10}$-net of $S_U.$ Lemma 3.5 of \cite{Vers} guarantees that $|\mathcal{A}| \leq 21^s.$ Next, notice that for every $\textbf{v} \in S_u,$ there exists some $u_i \in \mathcal{A}$ such that $\|\Delta \textbf{v}\| \leq \frac{1}{10},$ where  $\Delta \textbf{v} = \textbf{v} - u_i.$ Then it holds
$$ \gamma:= \sup_{\textbf{v} \in S_U}|\textbf{v}^{\top}\mathcal{D}_C \textbf{v}| \leq \max_{i} |u_i^{\top} \mathcal{D}_C u_i| + 2 \sup_{\textbf{v} \in S_U} | \max_{i} u_i^{\top} \mathcal{D}_C (\Delta \textbf{v})| + \sup_{v \in S_U}| (\Delta \textbf{v})^{\top} \mathcal{D}_C (\Delta \textbf{v})|.$$
Next, we use the fact that $10 (\Delta \textbf{v}) \in S_U$ which gives us in consequence
$$ \sup_{\textbf{v} \in S_U} | (\Delta \textbf{v})^{\top} \mathcal{D}_C (\Delta \textbf{v})| \leq \frac{1}{100}\gamma$$
and 
\begin{align*}
  &2 \sup_{\textbf{v} \in S_U} | \max_{i} u_i^{\top} \mathcal{D}_C (\Delta \textbf{v})|   \\ &\leq \frac{1}{10} \sup_{\textbf{v} \in S_U} |(u_i + 10 \Delta \textbf{v})^{\top}\mathcal{D}_C (u_i + 10 \Delta \textbf{v})| + \frac{1}{10}\sup_{\textbf{v} \in S_U} |u_i \mathcal{D}_C u_i| + \frac{1}{10}\sup_{v \in S_U}|(10 \Delta \textbf{v})^{\top} \mathcal{D}_C (10 \Delta \textbf{v})| \\ & \leq \frac{4}{10} \gamma + \frac{1}{10}\gamma + \frac{1}{10} \gamma
\end{align*}
which implies that
$$ \gamma \leq 3 \max_i |u_i^{\top} \mathcal{D}_C u_i|.$$
Now, we take an union bound over all $u_i \in \mathcal{A}$ and combine it with inequality \eqref{ineqH} from Proposition \ref{PropH4}. Thus,

$$\mathbb{P}\left( \sup_{\textbf{v} \in S_U} |\textbf{v}^{\top} \mathcal{D}_C \textbf{v}| \geq x\right) \leq 2 d \exp (-TH_0(x) + s \log 21).$$ 
Next, we take another union bound over $ {d^2 \choose s} \leq \left( \frac{e d^2}{s}\right)^s $
choices of $U$. Thus,
\bee
\P\left( \sup_{\textbf{v}\in\R^{d^2}\setminus\{0\}:~ \|\textbf{v}\|_0\leq s}\frac{|\textbf{v}^{\top} \mathcal{D}_C \textbf{v}|}{\|\textbf{v}\|_2^2} \geq 3x \right) \leq 2 d \Big(\frac{21ed^2}{s}\Big)^s   \exp(-T H_0(x)) ,
\eee
which yields the proof.
\end{proof}

\bibliographystyle{chicago}
 
\end{document}